\documentclass{article}

\usepackage{amsfonts}
\usepackage{amsthm}
\usepackage{amsmath}
\usepackage{amssymb}
\usepackage{mathtools} %for prescript
\usepackage{authblk} %for affil
\usepackage{xcolor}
\usepackage{soul} %for strikethrough
\usepackage{mathrsfs}

\usepackage[top=1in, bottom=1in, left=1in, right=1in]{geometry}

\newtheorem{theorem}{Theorem}[section]
\newtheorem{lemma}{Lemma}[section]
\newtheorem{corollary}{Corollary}[section]

\theoremstyle{definition}
\newtheorem{definition}{Definition}[section]

\theoremstyle{remark}
\newtheorem{remark}{Remark}[section]

\numberwithin{equation}{section}

\begin{document}
	
	%	\begin{frontmatter}
	
	\title{A new technique for solving Sobolev type fractional multi-order evolution equations}
	
	%%this line removes the date, but space is still left for it;
	%if used, remove the \vspace{-1cm}
	\date{}
	
	%this gives the date in the form Mon 30 Jan 2012, 8:57pm;
	%if used, retain the \vspace{-1cm}
	%\date{\shortdayofweekname{\day}{\month}{\year}{ }\mydate\today}

	%\author[]{}
    %\ead{arzu.ahmadova@emu.edu.tr}
	\author[]{Nazim I. Mahmudov\thanks{ Corresponding author.  Email: \texttt{nazim.mahmudov@emu.edu.tr}}}
	\author[]{Arzu Ahmadova \thanks{Email: \texttt{arzu.ahmadova@emu.edu.tr}}}
	\author[]{Ismail T. Huseynov \thanks{Email: \texttt{ismail.huseynov@emu.edu.tr}}}
	%	\ead{nazim.mahmudov@emu.edu.tr}

	\affil {Eastern Mediterranean University, Department of Mathematics, Gazimagusa, 99628, TRNC, Mersin 10, Turkey  }
	
	% Latex won't make the title unless told:
	
	\maketitle
	
	%%to remove the space left for date, use:

	\begin{abstract}
		\noindent A strong inspiration for studying Sobolev type  fractional evolution equations comes from the fact that have been verified to be useful tools in the modeling of many physical processes.  We introduce a novel technique for solving Sobolev type fractional evolution equations with multi-orders in a Banach space. We propose a new Mittag-Leffler type function which is generated by linear bounded operators and investigate their properties which are productive for checking the candidate solutions for multi-term fractional differential equations. Furthermore, we propose an exact analytical representation of solutions for multi-dimensional fractional-order dynamical systems with nonpermutable and permutable matrices.
	\end{abstract}

	%	\end{keyword} 
	 
	 	
   	\section{Introduction}\label{Sec:intro}
   	Fractional differential equations (FDEs) are a generalization of the classical ordinary and partial differential equations, in which the order of differentiation is permitted to be any real (or even complex) number, not only a natural number.
   	FDEs containing not only one fractional derivative but also more than one fractional derivative are intensively studied in many physical processes \cite{Kilbas,Podlubny,Samko}. Many authors demonstrate two essential mathematical ways to use this idea: multi-term equations \cite{Ahmadova-Mahmudov,Bazhlekova E.,Hilfer-Luchko-Tomovski,Luchko and Gorenflo, Luchko} and multi-order systems \cite{A-H-F-M,ismail-arzu}.

   	Multi-term FDEs have been studied due to their applications in modelling, and solved using various mathematical methods. Finding the solution to these equations is an interesting and challenging subject that attracted many scientists in the last decades. Up to now, various analytical and computational techniques have been investigated to find the solution of multi-term FDEs, of which we mention a few as follows. Luchko and several collaborators \cite{Hilfer-Luchko-Tomovski,Luchko and Gorenflo,Luchko} used the method of operational calculus to solve multi-order FDEs with different types of fractional derivatives. In the realm of ordinary differential equations, Mahmudov and other collaborators  \cite{Ahmadova-Mahmudov,Mahmudov-Huseynov-Aliev-Aliev} have derived an analytical representation of solutions for special cases of fractional differential equations with multi-orders, namely: Langevin and Bagley-Torvik equations involving scalar coefficients and permutable matrices by using Laplace transform method and fractional analogue of variation of constants formula, respectively, while the other authors \cite{Pak-Choi-Sin-Ri} have solved multi-term differential equations in Riemann-Liouville's sense with variable coefficients applying a new method to construct analytical solutions. 
   	
   	Several results have been investigated on solving multi-dimensional time-delay deterministic and stochastic systems with permutable matrices \cite{arzu-mahmudov,Huseynov and Mahmudov} in classical and fractional senses. In \cite{diblik}, Diblik et al. have considered inhomogeneous system of linear differential equations of second order with multiple different delays \& pairwise permutable matrices and represented a solution of corresponding initial value problem by using matrix polynomials. Khusainov and other collaborators in \cite{Khusainov-permutable,Khusainov-Shuklin} have proposed exact analytical representations of linear autonomous time-delay systems with commutative matrices. Pospisil \cite{Pospisil} has introduced a representation of the solution of delayed differential equation, the condition that the linear parts are given by pairwise permutable matrices. In \cite{Medved-Pospisil}, Medved and Pospisil have tackled this strong condition (commutativity of matrices) and derived a representation of solutions of functional differential equations with nonconstant coefficients and variable delays. Recently, Mahmudov
   	\cite{mahmudov3} has introduced a fractional analogue of delayed matrices cosine and sine in the commutative case i.e., $AB=BA$ to solve the sequential Riemann-Liouville type linear time-delay systems whilst Liang et al. \cite{liang} have also obtained an explicit solution of the differential equation with pure delay and sequential Caputo type fractional derivative.
   	
   	However, there are a few papers involving non-permutable matrices which are recently studied fractional time-continuous \cite{mahmudov1} and discrete \cite{mahmudov2}  systems with a constant delay using recursively defined matrix-equations, and also delayed linear difference equations \cite{Nmahmudov} applying $\mathscr{Z}$-transform technique by Mahmudov.

   	 Meanwhile, Sobolev type evolution equations and their fractional-order analogues have attracted a great deal of attention from applications' point of view  and studied by several authors \cite{Balachandran-1, Balachandran-2,Feckan-1, Wang-1,Wang-2} in recent decades. In \cite{Balachandran-1}, Balachandran and Dauer have derived sufficient conditions for controllability of partial functional differential systems of Sobolev type in a Banach space by using compact semigroups and Schauder's fixed point theorem. Moreover, Balachandran et al. \cite{Balachandran-2} have considered existence results of solutions for nonlinear impulsive integrodifferential equations of Sobolev type with nonlocal conditions via Krasnoselkii's fixed point technique. In terms of fractional differential equations, Wang et al. \cite{Wang-1} have investigated controllability results of Sobolev type fractional evolution equations in a seperable Banach space by using the theory of propagation family and contraction mapping principle. In addition, Feckan et al. \cite{Feckan-1} have presented controllability of fractional functional evolution sytems of Sobolev type with the help of new characteristic solution operators and well-known Schauder's fixed point approach. In addition, Mahmudov \cite{Mahmudov-Sobolev} have considered approximate controllability results for a class of fractional evolution equations of Sobolev type by using fixed point approach. In \cite{Wang-2}, Wang and Li have discussed stability analysis of fractional evolution equations of Sobolev type in Ulam-Hyers sense. In \cite{Chang-1}, Chang et al. have studied the asymptotic behaviour of resolvent operators of Sobolev type and their applications to the existence and uniqueness of mild solutions to fractional functional evolution equations in Banach spaces. Vijayakumar et al. \cite{Vijakumar-1} have presented approximate controllability results for Sobolev type time-delay differential systems of fractional-order in Hilbert spaces.

   	To the best of our knowledge, the fractional evolution equations of Sobolev type with non-permutable operators and two independent fractional orders of differentiation $\alpha$ and $\beta$ which are  assumed to be in the interval $(1,2]$ and $(0,1]$, respectively are an untreated topic in the present literature.
   	Thus, motivated by the above research works, we consider the following Cauchy problem for fractional evolution  equation of Sobolev type with orders  $1<\alpha\leq 2$ and $0<\beta \leq 1$ on $\mathbb{J}\coloneqq[0,T]$:
   	
   	\begin{align}\label{mtde}
   	\begin{cases}
   	 \left( \prescript{C}{}{D^{\alpha}_{0+}}Ey\right) (t) -A_{0} \left( \prescript{C}{}{D^{\beta}_{0+}}y\right) (t)=B_{0}y(t)+g(t), \quad t>0,\\
   	Ey(0)=\eta , \quad Ey^{\prime}(0)=\tilde{\eta},
   	\end{cases}
   	\end{align}
   	 where $\prescript{C}{}{D^{\alpha}_{0+}}$ and $ \prescript{C}{}{D^{\beta}_{0+}}$ Caputo fractional differential operators of orders $1<\alpha\leq2$ and $0<\beta\leq1$, respectively, with the lower limit zero, the operators $E: D(E)\subset X \to Y$, $A_{0}:D(A_{0})\subset X \to Y$ and $B_{0 }:D(B_{0})\subset X \to Y$ are linear, where $X$ and $Y$ are Banach spaces, $y(\cdot)$ is a $X$-valued function on $\mathbb{J}$, i.e., $y(\cdot):\mathbb{J}\to X$ and $\eta, \hat{\eta}\in Y$. In addition, $g(\cdot): \mathbb{J}\to Y$ is a continuous function. The domain $D(E)$ of $E$ becomes a Banach space with respect to $\|y\|_{D(E)}=\|Ey\|_{Y}, y \in D(E)$.
   
   The main idea is that under the hypotheses $(H_{1})$-$(H_{4})$ we transform Sobolev type fractional multi-term evolution equation with linear operators \eqref{mtde} to fractional-order evolution equation with multi-orders and linear bounded operators \eqref{mtde-1}. Secondly, we solve fractional evolution equation with nonpermutable linear bounded operators by using Laplace transform technique which is used as a necessary tool for solving and analyzing fractional-order differential equations and systems in \cite{Ahmadova-Mahmudov},\cite{Huseynov and Mahmudov},\cite{kexue},\cite{Sabatier-Moze-Farges}.
   Then we  propose exact analytical representation of a mild solution of \eqref{mtde-1} and \eqref{mtde},  respectively with the help of new defined Mittag-Leffler function which is expressed via linear bounded operators by removing the exponential boundedness of  a forced term $g(\cdot)$ and $\left( \prescript{C}{}{D}_{0+}^{\beta}x\right)(\cdot)$ for $\beta\in(0,1]$ (or $\left( \prescript{C}{}{D}_{0+}^{\alpha}x\right)(\cdot)$ for $\alpha\in(1,2])$ in both cases: with nonpermutable and permutable linear operators $A,B\in\mathscr{B}(Y)$.
   
   The structure of this paper contains important improvement in the theory of Sobolev type fractional multi-term evolution equations and is outlined as below. Section \ref{sec:2} is a preparatory section where we recall main definitions and results from fractional calculus, special functions and fractional differential equations. In Section \ref{sec:3}, we establish a new  Mittag-Leffler type function which is generated by linear bounded operators via a double infinity series and investigate some necessary properties of this function which are accurate tool for testing the candidate solutions of fractional-order dynamical equations. We also investigate that $Q_{k,m}^{A,B}$ with nonpermutable linear operators $A,B\in \mathscr{B}(Y)$ is a generalization of well-known Pascal's rule binomial coefficients.  Moreover, we introduce the sufficient conditions for exponential boundedness of \eqref{mtde-1} to guarantee the existence of Laplace integral transform of equation \eqref{mtde-1}. Then we solve multi-order fractional evolution equations  \eqref{mtde} and \eqref{mtde-1} with the help of Laplace integral transform. Meanwhile, we tackle this strong condition and verify that the sufficient conditions can be omitted easily. Section \ref{sec:4} deals with an analytical representation of a mild solution to Sobolev type evolution equations with commutative linear bounded operators. In addition, we propose exact solutions for multi-dimensional multi-term fractional dynamical systems with commutative and noncommutative matrices.  In Section \ref{sec:concl} we discuss our main contributions of this paper and future research work.

   	\section{Preliminary concept}\label{sec:2}

   	We embark on this section by briefly presenting some notations and definition fractional calculus and fractional differential equations \cite{Kilbas,Samko,Podlubny}  which are used throughout the paper. 
   	
   	 Let $\mathbb{C}^{2}\left(\mathbb{J},X\right)\coloneqq\left\lbrace y\in \mathbb{C}\left(\mathbb{J},X\right) :  y^{'},y^{''}\in\mathbb{C}\left(\mathbb{J},X\right) \right\rbrace $ denote the Banach space of functions $y(t)\in X$ for $t\in \mathbb{J}$ equipped with a norm $\|y\|_{\mathbb{C}^{2}(\mathbb{J},X)}=\sum\limits_{i=0}^{2}\sup\limits_{t\in \mathbb{J}}\|y^{(i)}(t)\|$. The space of all bounded linear operators from $X$ to $Y$ is denoted by $\mathscr{B}(X,Y)$ and $\mathscr{B}(Y,Y)$ is written as $\mathscr{B}(Y)$.

  \begin{definition}\label{RLI}\cite{Kilbas,Samko,Podlubny}
  	The fractional integral of  order $\alpha > 0 $   for a function  $\ g\in  \left( [0,\infty), \mathbb{R}\right) $  is defined by
  	\begin{equation}
  	\prescript{}{}(I^{\alpha}_{0^{+}}g)(t)=\frac{1}{\Gamma(\alpha)}\int\limits_0^t(t-s)^{\alpha-1}g(s)\,\mathrm{d}s \\, \quad t>0,
  	\end{equation}
  	where $\Gamma(\cdot)$ is the well-known Euler's gamma function.% \cite{Rainville} which is given by 
  	%\begin{equation} \label{gamma}
  	%\Gamma(\alpha) = \int\limits_{0}^{\infty}t^{\alpha-1}e^{-t}\mathrm{d}t.
  	%\end{equation}
  \end{definition}
  %\begin{definition} \label{beta} \cite{Rainville}
  %	The beta function is defined by  the definite integral:
  	%\begin{equation}
  	%\mathbf{B}(\alpha,\beta)=\int\limits_{0}^{1}t^{\alpha-1}(1-t)^{\beta-1}\mathrm{d}t, \quad \alpha>0, \beta>0. 
  	%\end{equation}
  %\end{definition}
  %In addition, the beta function can be expressed via gamma functions as follows:
  %\begin{equation*}
  %\mathbf{B}(\alpha,\beta)=\frac{\Gamma(\alpha)\Gamma(\beta)}{\Gamma(\alpha+\beta)}, \quad \alpha>0, \beta>0. 
  %\end{equation*}
  \begin{definition}\label{RLD} \cite{Kilbas,Samko,Podlubny}
  	The Riemann-Liouville fractional derivative of order $n-1<\alpha\leq n$, $n \in \mathbb{N}$ for a function  $\ g\in  \left( [0,\infty), \mathbb{R}\right) $ is  defined by 
  	\begin{equation}
  	(\prescript{RL}{}D^{\alpha}_{0^{+}}g)(t)=\frac{1}{\Gamma(n-\alpha)}\left( \frac{d}{dt}\right) ^{n}\int\limits_0^t(t-s)^{n-\alpha-1}g(s)\,\mathrm{d}s, \quad t>0,
  	\end{equation}
  	where the function $g(\cdot)$ has absolutely continuous derivatives up to order $n$.
  \end{definition}

 The following theorem and its corollary is regarding fractional analogue of the eminent Leibniz integral rule for general order $\alpha \in (n-1,n], n \in \mathbb{N}$ in Riemann-Liouville's sense which is more productive tool for the testing particular solution of inhomogeneous linear multi-order fractional differential equations with variable and constant coefficients is considered by Huseynov et al. \cite{Leibniz}. 
\begin{theorem}\label{thm-class}
	Let the function $K:J\times J\to\mathbb{R}$ be such that the following assumptions are fulfilled:
	
	(a) For every fixed $t\in J$, the function $\hat{K}(t,s)=\prescript{RL,t}{}{D^{\alpha-1}_{s+}}K(t,s)$ is measurable on $J$ and integrable on $J$ with respect to some $t^{*}\in J$;
	
	(b) The partial derivative $\prescript{RL,t}{}{D^{\alpha}_{s+}}K(t,s)$ exists for every interior point $(t, s) \in \hat{J} \times \hat{J}$;
	
	(c) There exists a non-negative integrable function $g$ such that  $\left|\prescript{RL,t}{}{D^{\alpha}_{s+}}K(t,s)\right| \leq g(s)$ for every interior point $(t, s) \in \hat{J}\times \hat{J}$;

	(d) The derivative $\frac{d^{l-1}}{dt^{l-1}}\lim\limits_{s\to t-0}\prescript{RL,t}{}{D^{\alpha-l}_{s+}}K(t,s)$, $l=1,2,\ldots,n$ exists for every interior point $(t, s) \in \hat{J} \times \hat{J}$.
	
	Then, the following relation holds true for fractional derivative in Riemann-Liouville sense under Lebesgue integration for any $t \in \hat{J}$:	
	\begin{equation}\label{thm-RL}
	\prescript{RL}{}{D^{\alpha}_{t_{0}+}}\int\limits_{t_{0}}^{t}K(t,s)\mathrm{d}s=\sum_{l=1}^{n}\frac{d^{l-1}}{dt^{l-1}}\lim\limits_{s\to t-0}\prescript{RL,t}{}{D^{\alpha-l}_{s+}}K(t,s)+\int\limits_{t_{0}}^{t}\prescript{RL,t}{}{D^{\alpha}_{s+}}K(t,s)\mathrm{d}s.  
	\end{equation}
	
	If we have $K(t,s)=f(t-s)g(s)$, $t_{0}=0$ and assumptions of Theorem \ref{thm-class} are fulfilled, then following equality holds true for convolution operator in Riemann-Liouville sense for any $n \in\mathbb{N}$:
	
	\begin{align}\label{Leibniz}
	\prescript{RL}{}{D^{\alpha}_{0+}}\int\limits_{0}^{t}f(t-s)g(s)\mathrm{d}s&=\sum_{l=1}^{n}\lim\limits_{s\to t-0} \prescript{RL,t}{}{D^{\alpha-l}_{s+}}f(t-s)\frac{d^{l-1}}{dt^{l-1}}\lim\limits_{s\to t-0}g(s)\nonumber\\&+\int\limits_{0}^{t}\prescript{RL,t}{}{D^{\alpha}_{s+}}f(t-s)g(s)\mathrm{d}s, \quad  t >0.
	\end{align}
	
where $\prescript{RL,t}{}{D^{\gamma}_{t_{0}+}}K(t,s)$ is a partial Riemann-Liouville fractional differentiation operator of order $\gamma > 0$ \cite{Kilbas}  with respect to $t$ of a function $K(t,s)$ of two variables with lower terminal $t_{0}$ and $J=[t_{0},T]$,  $\hat{J}=(t_{0},T)$.
\end{theorem}

In the special cases, Riemann-Riouville type differentiation under integral sign holds for convolution operator \cite{Leibniz}:

\begin{itemize}
	\item If $\alpha \in (0,1]$, then 
	\begin{align*}
	\prescript{RL}{}{D^{\alpha}_{0+}}\int\limits_{0}^{t}f(t-s)g(s)\mathrm{d}s&=\lim\limits_{s\to t-0} \prescript{RL,t}{}{D^{\alpha-1}_{s+}}f(t-s)\lim\limits_{s\to t-0}g(s)\nonumber\\&+\int\limits_{0}^{t}\prescript{RL,t}{}{D^{\alpha}_{s+}}f(t-s)g(s)\mathrm{d}s, \quad  t >0 ;
	\end{align*}
	\item  If $\alpha \in (1,2]$, then 
	\begin{align*}
	\prescript{RL}{}{D^{\alpha}_{0+}}\int\limits_{0}^{t}f(t-s)g(s)\mathrm{d}s&=\lim\limits_{s\to t-0} \prescript{RL,t}{}{D^{\alpha-1}_{s+}}f(t-s)\lim\limits_{s\to t-0}g(s)\nonumber\\&+\lim\limits_{s\to t-0} \prescript{RL,t}{}{D^{\alpha-2}_{s+}}f(t-s)\lim\limits_{s\to t-0}g(s)+\int\limits_{0}^{t}\prescript{RL,t}{}{D^{\alpha}_{s+}}f(t-s)g(s)\mathrm{d}s, \quad  t >0.
	\end{align*}
\end{itemize}

\begin{definition}\label{CD}\cite{Kilbas,Podlubny} 
  	The Caputo fractional derivative  of  order, $n-1<\alpha\leq n$, $n \in \mathbb{N}$ for a function  $\ g\in  \left( [0,\infty), \mathbb{R}\right) $  is defined by
  	\begin{equation}
  	(\prescript{C}{}D^{\alpha}_{0^{+}}g)(t)=\frac{1}{\Gamma(n-\alpha)}\int\limits_0^t(t-s)^{n-\alpha-1}\left(\frac{d}{ds}\right)^{n}g(s)\,\mathrm{d}s \\, \quad t>0,
  	\end{equation}
  	where the function $g(\cdot)$ has absolutely continuous derivatives up to order $n$.
  \end{definition}

\begin{definition}\label{RLC}\cite{Kilbas,Podlubny} 
	The relationship between Caputo and Riemann-Liouville fractional differential operators of  order $n-1<\alpha\leq n$, $n \in \mathbb{N}$ for a function  $\ g\in  \left( [0,\infty), \mathbb{R}\right) $ is defined by
	\begin{equation}\label{relation}
	(\prescript{C}{}D^{\alpha}_{0^{+}}g)(t)=	\prescript{RL}{}D^{\alpha}_{0^{+}}\left( g(t)-\sum_{k=0}^{n-1}\frac{t^{k}}{k!}g^{(k)}(0)\right)  \\, \quad t>0,
	\end{equation}
	where the function $g(\cdot)$ has absolutely continuous derivatives up to order $n$.
\end{definition}
\begin{remark}
	If $g(\cdot)$ is an abstract function with values in $X$, then the integrals which appear in Definition \ref{RLI},  \ref{RLD}, \ref{CD} and \ref{RLC} are taken in Bochner's sense.
\end{remark}

  The Laplace transform of the Caputo's fractional differentiation operator \cite{Kilbas} is defined by
  \begin{equation} \label{lapCaputo}
  \mathscr{L} \left\lbrace (\prescript{C}{}D^{\alpha}_{0^{+}}g)(t)\right\rbrace(s) =s^{\alpha}G(s)-\sum_{k=1}^{n}s^{\alpha-k}g^{(k-1)}(0), \quad n-1<\alpha\leq n, \quad n\in\mathbb{N},
  \end{equation}
  where $G(s)=\mathscr{L} \left\lbrace g(t) \right\rbrace(s)$. 
 
  In the particular cases, the Laplace integral transform of the Caputo fractional derivative is:
 	
 	\begin{itemize}\label{remLap}
 		\item If $\alpha \in (0,1]$, then 
 		\begin{equation*}
 		\mathscr{L}\left\lbrace  \left( \prescript{C}{}D^{\alpha}_{0^{+}}x\right) (t)\right\rbrace (s)=s^{\alpha}X(s)-s^{\alpha-1}x(0) ;
 		\end{equation*}
 		\item  If $\alpha \in (1,2]$, then 
 		\begin{equation*}
 		\mathscr{L}\left\lbrace  \left( \prescript{C}{}D^{\alpha}_{0^{+}}x\right) (t)\right\rbrace (s)=s^{\alpha}X(s)-s^{\alpha-1}x(0)-s^{\alpha-2}x^{\prime}(0) ,
 		\end{equation*}
 	\end{itemize} 
  
  where $X(s)=\mathscr{L} \left\lbrace x(t) \right\rbrace(s)$.

  \begin{lemma} [\label{sumA}\cite{Yosida}]
  	Suppose that A is linear bounded operator defined on the Banach space $X$ and assume that $\|A\| < 1$. Then, $(I-A)^{-1}$ is linear bounded on $X$ and
  	\begin{equation}\label{operator}
  	(I-A)^{-1}=\sum_{k=0}^{\infty}A^{k}.
  	\end{equation}
  \end{lemma}
  
  The following well-known generalized Gronwall inequality which plays an important role in the qualitative analysis of the solutions to fractional differential equations is stated and proved in \cite{henry,ye} for $\beta>0$. In particular case, if $\beta=1$, then the following relations hold true:
  	\begin{theorem}\label{thm-1} 
  	Suppose $a(t)$ is a nonnegative function locally integrable on $0 \leq t < T$ (some $T\leq +\infty$), $b(t)$ is a nonnegative, nondecreasing continuous function defined on $0\leq t<T$, $|b(t)|\leq M$, ($M$ is a positive constant) and suppose $u(t)$ is a nonnegative and locally integrable on  $0 \leq t < T$ with
  	\begin{equation*}
  	u(t)\leq a(t)+ b(t)\int\limits_{0}^{t}u(s)ds,
  	\end{equation*}
  	on this interval; then
  	\begin{equation*}
  	u(t)\leq a(t)+ b(t)\int\limits_{0}^{t}\exp\left( b(t)(t-s)\right) a(s)ds, \quad 0 \leq t < T.
  	\end{equation*}
  \end{theorem} 
  \begin{corollary}
  	Under the hypothesis of Theorem \ref{thm-1}, let $a(t)$ be a nondecreasing function on $[0,T)$.
  	Then
  	\begin{equation}\label{coroll}
  	u(t)\leq a(t) \exp\left( b(t)t\right) , \quad 0 \leq t < T.
  	\end{equation}
  \end{corollary} 
 
  The Mittag-Leffler function is a natural generalization of the exponential function, first proposed as a single parameter function of one variable by using an infinite series \cite{ML}. Extensions to two or three parameters are well known and thoroughly studied in textbooks such as \cite{Gorenflo}. Extensions to two or several variables have been studied more recently \cite{ ismail-arzu,A-H-F-M,fernandez-kurt-ozarslan,saxena-kalla-saxena}.
  \begin{definition}[\cite{ML}] \label{Def:ML}
  	The classical Mittag-Leffler function is defined by
  	\begin{equation}\label{ML1}
  	E_{\alpha}(t)= \sum_{k=0}^{\infty}\frac{t^{k}}{\Gamma(k \alpha +1)}, \quad  \alpha>0, \quad t\in\mathbb{R}.
  	\end{equation}

  	The two-parameter Mttag-Leffler function  \cite{Wiman} is given by
  	\begin{equation}\label{ML-2}
  	E_{\alpha,\beta}(t)= \sum_{k=0}^{\infty}\frac{t^{k}}{\Gamma(k \alpha +\beta)}, \quad  \alpha>0, \quad \beta\in\mathbb{R}, \quad t\in\mathbb{R}.
  	\end{equation}
  	The three-parameter Mittag-Leffler function \cite{Prabhakar} is determined by
  	\begin{equation}
  	E_{\alpha,\beta}^{\gamma}(t)= \sum_{k=0}^{\infty}\frac{(\gamma)_k}{\Gamma(k \alpha +\beta)}\frac{t^{k}}{k!}, \quad  \alpha >0, \quad \beta,\gamma\in\mathbb{R}, \quad  t\in\mathbb{R},
  	\end{equation}
  	where $(\gamma)_k$ is the Pochhammer symbol denoting $\frac{\Gamma(\gamma+k)}{\Gamma(\gamma)}$. These series are convergent, locally uniformly in $t$, provided the $\alpha>0$ condition is satisfied. It is important to note that
  	\[
  	E_{\alpha,\beta}^1(t)=E_{\alpha,\beta}(t),\quad E_{\alpha,1}(t)=E_{\alpha}(t),\quad E_1(t)=\exp(t).
  	\]
  \end{definition}

   	\begin{lemma}[\cite{Prabhakar}] \label{Lem:PrabhLap}
   		The Laplace transform of the three-parameter Mittag-Leffler function is given by
   		\begin{equation}\label{ad}
   		\mathscr{L}\left\{t^{\beta-1}E_{\alpha,\beta}^{\gamma}(\lambda t^{\alpha})\right\}(s)=s^{-\beta}\left(1-\lambda s^{-\alpha}\right)^{-\gamma},
   		\end{equation}
   		where $\alpha>0$,\quad$\beta,\gamma,\lambda \in \mathbb{R}$ and $Re(s)>0$.
   	\end{lemma}
   	\begin{definition}\cite{fernandez-kurt-ozarslan}
   	A bivariate Mittag-Leffler type function which is a particular case of multivariate Mittag-Leffler function \cite{Luchko and Gorenflo} is defined by
   \begin{equation}\label{bivtype}
   	E_{\alpha,\beta,\gamma}(\lambda_{1}t^{\alpha},\lambda_{2}t^{\beta})=\sum_{k=0}^{\infty}\sum_{m=0}^{\infty}\binom{k+m}{m}\frac{\lambda_{1}^k\lambda_{2}^mt^{k\alpha+m\beta}}{\Gamma(k\alpha+m\beta+\gamma)}, \quad  \alpha,\beta>0, \quad  \gamma\in \mathbb{R}.
   \end{equation}
   	\end{definition}

	\section{A representation of mild solution to \eqref{mtde} with non-permutable linear operators}\label{sec:3}
In this section, we consider the Cauchy problem for fractional evolution equation of Sobolev type in a Banach spaces.
Firstly, we introduce the following hypotheses on the linear operators  $A_{0}$, $B_{0}$ and $E$:

$(H_{1})$: $A_{0}$ is a closed operator;

$(H_{2})$: $B_{0}$ is a bounded operator;

$(H_{3})$: $D(E)\subset D(A_{0})$ and $E$ is bijective;

$(H_{4})$: A linear operator $E^{-1}: Y \to D(E)\subset X$ is compact.
	
It is important to stress out that $(H_{4})$ implies  $E^{-1}$ is bounded. Furthermore, $(H_{4})$ also implies that $E$ is closed since the fact:
$E^{-1}$ is closed and injective, then its inverse is also closed. It comes from the closed graph theorem, we acquire the boundedness of the linear operator $ A\coloneqq A_{0}E^{-1}: Y \to Y$. Furthermore, $ B\coloneqq   B_{0}E^{-1}: Y \to Y$ is a linear bounded operator since $E^{-1}$ and $B_{0}$ are bounded.

 Obviously, the substitution $Ey(t)=x(t)$ is equivalent to  $y(t)=E^{-1}x(t)$. The central idea is that applying the substitution $y(t)=E^{-1}x(t)$, under the hypotheses $(H_{1})-(H_{4})$, we transform the Sobolev type fractional-order evolution system \eqref{mtde} to the  following multi-term evolution system with linear bounded operators $A,B\in\mathscr{B}(Y)$:
 
\begin{equation}\label{mtde-1}
\begin{cases}
\left(\prescript{C}{}{D^{\alpha}_{0+}}x\right) (t) -A\left( \prescript{C}{}{D^{\beta}_{0+}}x\right) (t)=Bx(t)+g(t), \quad t>0,\\
x(0)=\eta , \quad x^{\prime}(0)=\tilde{\eta},
\end{cases}
\end{equation}
where $x(\cdot):\mathbb{J}\to Y$ and $\eta,\tilde{\eta}\in Y$.

This signifies that a mild solution of the Cauchy problem for Sobolev type multi-term fractional evolution equation \eqref{mtde} is the multiplication of $E^{-1}\in \mathscr{B}(Y)$ and  the solution of an initial value problem for fractional evolution equation with multi-orders and linear bounded operators \eqref{mtde-1}.
\begin{remark}
Alternatively, we can modify the assumptions which are given  above in a similar way:

$(H'_1)$: $A_{0}$ is a bounded operator;

$(H'_2)$: $B_{0}$ is a closed operator;

$(H'_3)$: $D(E)\subset D(B_{0})$ and $E$ is bijective;

$(H'_4)$: $E^{-1}: Y\to D(E)\subset X$ is compact.

It follows from the closed graph theorem $B\coloneqq B_{0}E^{-1}:Y \to Y$ is a linear bounded operator. Furthermore, $A\coloneqq A_{0}E^{-1}:Y\to Y$ is also a linear bounded operator since $A_{0}$ and $E^{-1}$ are bounded. In conclusion, under the assumptions $(H'_1)-(H'_4)$, the Sobolev type fractional multi-term evolution equation with initial conditions \eqref{mtde} is converted to the  fractional evolution system with linear bounded operators \eqref{mtde-1} by using the same transformation $y(t)=E^{-1}x(t)$.
\end{remark}

To get an analytical representation of the mild solution of \eqref{mtde-1}, first, we need to show that exponentially boundedness of $x(\cdot)$ and its Caputo derivatives $\left( \prescript{C}{}{D^{\alpha}_{0+}}x\right)(\cdot)$ ,$\left( \prescript{C}{}{D^{\beta}_{0+}}x\right)(\cdot)$ for $1<\alpha\leq2$ and $0<\beta\leq 1$, respectively.  To do this, we need to assume exponential boundedness for one of the given fractional differentiation operators and a forced term with the aid of following Theorem \ref{thm2}. 

	\begin{theorem} \label{thm2}
		Assume \eqref{mtde-1} has a unique continuous solution $x(t)$, if $g(t)$ is continuous  \& exponentially bounded  and $\left( \prescript{C}{}D^{\beta}_{0^{+}}x\right) (t)$ for $0<\beta\leq1$ is exponentially bounded on $\left[ 0, \infty\right)$, then $x(t)$ and its Caputo derivative $\left( \prescript{C}{}D^{\alpha}_{0^{+}}x\right ) (t)$ is exponentially bounded for $1<\alpha\leq2$ on $\left[ 0, \infty\right)$ and, thus, their Laplace transforms exist.
	\end{theorem}
	\begin{proof}
		Since $g(t)$ and $\left( \prescript{C}{}D^{\beta}_{0^{+}}x\right) (t)$ for $0<\beta\leq1$  is exponentially bounded, there exists positive constants $L,P,\delta$  and sufficient large $T$ such that $\|g(t)\| \leq L \exp(\delta t)$ and $ \|\left( \prescript{C}{}D^{\beta}_{0^{+}}x\right) (t)\|\leq P \exp(\delta t)$  for any $t \geq T$. 
		It is clear that the system \eqref{mtde-1} is equivalent to the following Volterra fractional integral equation of second-kind:
		\begin{align} \label{Volterra}
		x(t)&=\left( 1-\frac{At^{\alpha-\beta}}{\Gamma(\alpha-\beta+1)}\right)\eta+t\hat{\eta}+\frac{A}{\Gamma(\alpha-\beta)}\int\limits_{0}^{t}(t-r)^{\alpha-\beta-1}x(r)dr\nonumber\\&+ \frac{1}{\Gamma(\alpha)}\int\limits_{0}^{t}(t-r)^{\alpha-1}[Bx(r)+g(r)]dr.
		\end{align}
		This means that every solution of \eqref{Volterra} is also a solution of \eqref{mtde-1} and vice versa. For $t \geq T$, \eqref{Volterra} can be expressed as 
		\begin{align*}
		x(t)&=\left( 1-\frac{At^{\alpha-\beta}}{\Gamma(\alpha-\beta+1)}\right)\eta+t\hat{\eta}+\frac{A}{\Gamma(\alpha-\beta)}\int\limits_{0}^{T}(t-r)^{\alpha-\beta-1}x(r)dr\nonumber\\&+ \frac{1}{\Gamma(\alpha)}\int\limits_{0}^{T}(t-r)^{\alpha-1}[Bx(r)+g(r)]dr+\frac{A}{\Gamma(\alpha-\beta)}\int\limits_{T}^{t}(t-r)^{\alpha-\beta-1}x(r)dr\nonumber\\&+ \frac{1}{\Gamma(\alpha)}\int\limits_{T}^{t}(t-r)^{\alpha-1}[Bx(r)+g(r)]dr.
		\end{align*}
		In view of hypotheses of Theorem \ref{thm2}, the solution $x(t), (x(0)=\eta, \quad x'(0)=\hat{\eta})$ is unique and continuous on $\left[ 0, \infty\right)$, then $Ax(t)$ and $Bx(t)+g(t)$ are bounded on $[0, T]$, namely:
		\begin{equation*}
		\exists M >0 \quad \text{s.t.} \quad \|Ax(t)\| \leq M, \quad \forall t \in [0, T],
		\end{equation*}
		and
		\begin{equation*}
		\exists N >0 \quad \text{s.t.} \quad \|Bx(t)+g(t)\| \leq N, \quad \forall t \in [0, T].
		\end{equation*}
		We have  
		\begin{align*}
		\|x(t)\|&\leq\left(  1+\frac{\|A\|t^{\alpha-\beta}}{\Gamma(\alpha-\beta+1)}\right) \|\eta\|+t\|\hat{\eta}\|+\frac{M}{\Gamma(\alpha-\beta)}\int\limits_{0}^{T}(t-r)^{\alpha-\beta-1}dr\nonumber\\&+ \frac{N}{\Gamma(\alpha)}\int\limits_{0}^{T}(t-r)^{\alpha-1}dr+\frac{\|A\|}{\Gamma(\alpha-\beta)}\int\limits_{T}^{t}(t-r)^{\alpha-\beta-1}\|x(r)\|dr\nonumber\\&+ \frac{\|B\|}{\Gamma(\alpha)}\int\limits_{T}^{t}(t-r)^{\alpha-1}\|x(r)\|dr+\frac{1}{\Gamma(\alpha)}\int\limits_{T}^{t}(t-r)^{\alpha-1}\|g(r)\|dr.
		\end{align*}
		
	  Multiplying last inequality by $\exp(-\delta t)$ and note that 
		\begin{equation*}
	 \exp(-\delta t) \leq \exp(-\delta r), \quad r \in [T,t]  \quad  \text{and} \quad 	\exp(-\delta t) \leq \exp(-\delta T),\quad  \|g(t)\| \leq L\exp(\delta t), \quad  t \geq T.
		\end{equation*}
		
		\allowdisplaybreaks
		Using the aforementioned inequalities, we attain
		\begin{align*}
		&\|x(t)\|\exp(-\delta t)\leq \|\eta\| \exp(-\delta t)+\frac{\|A\|t^{\alpha-\beta}}{\Gamma(\alpha-\beta+1)}\|\eta\|\exp(-\delta t)\\&+t\|\eta\|\exp(-\delta t)+ \frac{M\exp(-\delta t)}{\Gamma(\alpha-\beta)}\int\limits_{0}^{T}(t-r)^{\alpha-\beta-1}dr\\
		&+\frac{N\exp(-\delta t)}{\Gamma(\alpha)}\int\limits_{0}^{T}(t-r)^{\alpha-1}dr+  \frac{\|A\|\exp(-\delta t)}{\Gamma(\alpha-\beta)}\int\limits_{T}^{t}(t-r)^{\alpha-\beta-1}\|x(r)\|dr\\ &+\frac{\|B\|\exp(-\delta t)}{\Gamma(\alpha)}\int\limits_{T}^{t}(t-r)^{\alpha-1}\|x(r)\|dr+\frac{\exp(-\delta t)}{\Gamma(\alpha)}\int\limits_{T}^{t}(t-r)^{\alpha-1}\|g(r)\|dr\\
		&\leq\|\eta\|\exp(-\delta T)+\frac{\|A\|t^{\alpha-\beta}}{\Gamma(\alpha-\beta+1)}\|\eta\|\exp(-\delta T)+t\|\eta\|\exp(-\delta T)\\& +\frac{M\exp(-\theta T)}{\Gamma(\alpha-\beta+1)}(t^{\alpha-\beta}-(t-T)^{\alpha-\beta})+\frac{N\exp(-\delta T)}{\Gamma(\alpha+1)}(t^{\alpha}-(t-T)^{\alpha})\\ &+\frac{\|A\|}{\Gamma(\alpha-\beta)}\int\limits_{T}^{t}(t-r)^{\alpha-\beta-1}\|x(r)\|\exp(-\delta r)dr\\
		&+\frac{\|B\|}{\Gamma(\alpha)}\int\limits_{T}^{t}(t-r)^{\alpha-1}\|x(r)\|\exp(-\delta r)dr\\&+\frac{L}{\Gamma(\alpha)}\int\limits_{T}^{t}(t-r)^{\alpha-1}\exp(\delta(r-t))dr\\
		&\leq\|\eta\| \exp(-\delta T)+\frac{\|A\|t^{\alpha-\beta}}{\Gamma(\alpha-\beta+1)}\|\eta\|\exp(-\delta T)\\&+t\|\eta\|\exp(-\delta T) +\frac{M\exp(-\delta T)}{\Gamma(\alpha-\beta+1)}T^{\alpha-\beta}+\frac{N\exp(-\delta T)}{\Gamma(\alpha+1)}T^{\alpha}\\&+\int\limits_{0}^{t}\left( \frac{\|A\|(t-r)^{\alpha-\beta-1}}{\Gamma(\alpha-\beta)}+\frac{\|B\|(t-r)^{\alpha-1}}{\Gamma(\alpha)}\right) \|x(r)\|\exp(-\delta r)dr\\
		&+\frac{L}{\Gamma(\alpha)}\int\limits_{0}^{t}(t-r)^{\alpha-1}\exp(-\delta(t-r))dr\\
		&\leq\|\eta\| \exp(-\delta T)+\frac{\|A\|t^{\alpha-\beta}}{\Gamma(\alpha-\beta+1)}\|\eta\|\exp(-\delta T)\\&+t\|\eta\|\exp(-\delta T) +\frac{M\exp(-\delta T)}{\Gamma(\alpha-\beta+1)}T^{\alpha-\beta}+\frac{N\exp(-\delta T)}{\Gamma(\alpha+1)}T^{\alpha}\\&+\left( \frac{\|A\|t^{\alpha-\beta-1}}{\Gamma(\alpha-\beta)}+\frac{\|B\|t^{\alpha-1}}{\Gamma(\alpha)}\right)\int\limits_{0}^{t} \|x(r)\|\exp(-\delta r)dr+\frac{L}{\delta^{\alpha}}, \quad t \geq T.
		\end{align*}
		Denote 
		\begin{equation*}
		\begin{cases*}
		a(t)=\frac{\|A\|t^{\alpha-\beta}}{\Gamma(\alpha-\beta+1)}\|\eta\|\exp(-\delta T)+t\|\eta\|\exp(-\delta T)+\|\eta\| \exp(-\delta T)\\ \hspace{+0.65cm}+\frac{M\exp(-\delta T)}{\Gamma(\alpha-\beta+1)}T^{\alpha-\beta}+\frac{N\exp(-\delta T)}{\Gamma(\alpha+1)}T^{\alpha}+\frac{L}{\delta^{\alpha}} ,\\
		b(t)= \frac{\|A\|t^{\alpha-\beta-1}}{\Gamma(\alpha-\beta)}+\frac{\|B\|t^{\alpha-1}}{\Gamma(\alpha)}, \\
		v(t)= \|x(t)\| \exp(-\delta t).
		\end{cases*}
		\end{equation*}
		Thus, we attain
		\begin{equation}
		v(t) \leq a(t)+b(t) \int\limits_{0}^{t}v(s)ds, \quad t \geq T.
		\end{equation}
		According to the Gronwall's inequality \eqref{coroll}, we have
		\begin{equation} \label{rt}
		v(t)\leq a(t) \exp(tb(t))\leq\exp(a(t)+tb(t)) . 
		\end{equation}
	
		Then, it yields from \eqref{rt} that 
		\begin{align*}
		\|x(t)\|\leq\exp(a(t)+tb(t)+\delta t), \quad t \geq T.
		\end{align*}
		Since $g(t)$ and $\left( \prescript{C}{}D^{\beta}_{0^{+}}x\right) (t)$ for $\beta\in(0,1]$ are exponentially bounded on $[0, \infty)$, from equation \eqref{mtde-1}, we acquire
		\begin{align*}
		\| \left( \prescript{C}{}D^{\alpha}_{0^{+}}x\right) (t)\| &\leq \|A\|\|\left( \prescript{C}{}D^{\beta}_{0^{+}}x\right) (t)\|+\|B\| \|x(t)\|+ \|g(t)\| \\
		&\leq\|A\| P \exp(\delta t) + \|B\|\exp(a(t)+tb(t)+\delta t) + L \exp(\delta t) \\
		&\leq \left(\|A\| P+ \|B\|+L\right)  \exp(a(t)+tb(t)+\delta t), \quad t \geq T.
		\end{align*}
		In other words, $\left(\prescript{C}{}D^{\alpha}_{0^{+}}x \right)(t)$ is also exponentially bounded, the Laplace integral transforms of $x(t)$ and its Caputo derivatives $\left(\prescript{C}{}D^{\alpha}_{0^{+}}x \right)(t)$, $\left(\prescript{C}{}D^{\beta}_{0^{+}}x \right)(t)$ exist for $ \alpha \in (1,2]$ and $\beta \in (0,1]$, respectively. The proof is complete.
	\end{proof}
	Alternatively, we can also use the following version of Theorem \ref{thm2}, for exponential boundedness of $x(\cdot)$ and its derivatives $\left( \prescript{C}{}D^{\alpha}_{0^{+}}x\right) (\cdot)$ ,$\left( \prescript{C}{}D^{\beta}_{0^{+}}x\right) (\cdot)$ of order $1<\alpha\leq2$ and $0<\beta\leq1$, respectively in Caputo's sense on $\left[ 0, \infty\right)$.
	\begin{theorem} \label{thm3}
		Assume \eqref{mtde-1} has a unique continuous solution $x(t)$, if $g(t)$ is continuous  \& exponentially bounded  and $\left( \prescript{C}{}D^{\alpha}_{0^{+}}x\right) (t)$ for $1<\alpha\leq2$ is exponentially bounded on $\left[ 0, \infty\right)$, then $x(t)$ and its Caputo derivative $\left( \prescript{C}{}D^{\beta}_{0^{+}}x\right ) (t)$ is exponentially bounded for $0<\beta\leq1$ on $\left[ 0, \infty\right)$ and, thus, their Laplace transforms exist.
	\end{theorem}
\begin{proof}
This proof is similar to the proof of Theorem \ref{thm2}. So, we omit it here.
\end{proof}
\begin{definition}
We define a new Mittag-Leffler  function $ \mathscr{E}_{\alpha,\beta,\gamma}^{A,B}(\cdot) :\mathbb{R}\to Y$ generated by nonpermutable linear bounded operators $A,B\in\mathscr{B}(Y)$ as follows:
\begin{equation}
\mathscr{E}_{\alpha,\beta,\gamma}^{A,B}(t)\coloneqq\sum_{k=0}^{\infty}\sum_{m=0}^{\infty}Q_{k,m}^{A,B}\frac{t^{k\alpha+m\beta}}{\Gamma(k\alpha+m\beta+\gamma)}, \quad \alpha,\beta>0, \quad \gamma \in \mathbb{R},
\end{equation}
where $Q_{k,m}^{A,B}\in\mathscr{B}(Y)$, $k,m\in\mathbb{N}_{0}$ is given by

\begin{equation}\label{important}
Q_{k,m}^{A,B}\coloneqq\sum_{l=0}^{k} A^{k-l}BQ_{l,m-1}^{A,B}, \quad  k,m\in\mathbb{N}, \qquad Q_{k,0}^{A,B}\coloneqq A^{k}, \quad k\in\mathbb{N}_{0}, \qquad Q_{0,m}^{A,B}\coloneqq B^{m}, \quad m\in\mathbb{N}_{0}.
\end{equation}
\end{definition}
A linear bounded operator $Q_{k,m}^{A,B}$ can be represented explicitly in Table \ref{tab:1}.
  \begin{table}[h!]
	\caption{Explicit representation of $Q_{k,m}^{A,B}$ for $r,s \in \mathbb{N}_{0}$ }
	\label{tab:1}
	\centering
	\begin{tabular}{ c c c c c c }
		\hline
		$Q_{k,m}^{A,B}$ &k=0 & k=1 & k=2 & \ldots & k=r \\[2ex] \hline
		$m=0$ & $I$  & $A$  & $A^{2}$  &\ldots& $A^{r}$ \\[2ex]\hline
		$m=1$ & $B$ & $AB+BA$  & $A^{2}B+ABA+BA^{2}$&  \ldots&$A^{r}B+\ldots+BA^{r}$\\[2ex]\hline
		$m=2$ & $B^{2}$  & $AB^{2}+BAB+B^{2}A$   &
		$\!\begin{aligned}
			&A^{2}B^{2}+ABAB+AB^{2}A\\
			&+BA^{2}B+BABA+B^{2}A^{2}
		\end{aligned}$ &\ldots& $A^{r}B^{2}+\ldots+B^{2}A^{r}$\\[2ex]\hline
	 \ldots & \ldots  & \ldots  & \ldots  &\ldots &\ldots\\[2ex]\hline
		$m=s$ &$B^{s}$  &$AB^{s}+\ldots+B^{s}A$  & $A^{2}B^{s}+\ldots+B^{s}A^{2}$  &\ldots& $A^{r}B^{s}+\ldots+B^{s}A^{r}$\\[2ex]\hline
	\end{tabular}
\end{table}

From the above table, it can be easily seen that, in the case of commutativity $AB = BA$, we have
$Q_{k,m}^{A,B}:=\binom{k+m}{m}A^{k}B^{m}$, $k,m\in\mathbb{N}_{0}$.

\begin{theorem}
	A linear operator $Q_{k,m}^{A,B}\in\mathscr{B}(Y)$ for $k,m\in\mathbb{N}_{0}$ has the following properties:

	$(i)$ $Q_{k,m}^{A,B}$, $k,m\in\mathbb{N}$ generalizes classical Pascal's rule for linear operators $A,B\in\mathscr{B}(Y)$ as follows:
	\begin{equation}\label{gen-Pascal}
	Q_{k,m}^{A,B}=AQ_{k-1,m}^{A,B}+BQ_{k,m-1}^{A,B}, \quad k,m \in \mathbb{N};
	\end{equation} 
	
	$(ii)$ If $AB=BA$, then we have
	\begin{equation}\label{commutative}
	Q_{k,m}^{A,B}=\binom{k+m}{m}A^{k}B^{m}, \quad k,m\in\mathbb{N}_{0}.
	\end{equation}
\end{theorem}
\begin{proof}
	
	$(i)$ By making use of the mathematical induction principle, we can prove \eqref{gen-Pascal} is true for all $k \in \mathbb{N}$. It is obvious that the relation \eqref{gen-Pascal} is true for $k=1$. With the help of \eqref{important} we obtain:
	%\begin{align*}
	%Q_{0,m}^{A,B}=BQ_{0,m-1}^{A,B}=B^{m},
	%\end{align*}
	%and
	\begin{align*}
	Q_{1,m}^{A,B}=\sum_{l=0}^{1}A^{1-l}BQ_{l,m-1}^{A,B}=ABQ_{0,m-1}^{A,B}+BQ_{1,m-1}^{A,B}=AQ_{0,m}^{A,B}+BQ_{1,m-1}^{A,B},
	\end{align*}
	where $Q_{0,m}^{A,B}=BQ_{0,m-1}^{A,B}$.
	
	Suppose that the formula \eqref{gen-Pascal} is true for $(k-1)\in\mathbb{N}$. Then, by applying definition \eqref{important} for $(k-1)$-th case, we prove  the statement \eqref{gen-Pascal} for $k\in\mathbb{N}$ as below:
	\begin{align*}
	&Q_{k,m}^{A,B}=\sum_{l=0}^{k}A^{k-l}BQ_{l,m-1}^{A,B}=\sum_{l=0}^{k-1}A^{k-l}BQ_{l,m-1}^{A,B}+BQ_{k,m-1}^{A,B}\\&=A\sum_{l=0}^{k-1}A^{k-l-1}BQ_{l,m-1}^{A,B}+BQ_{k,m-1}^{A,B}=AQ_{k-1,m}^{A,B}+BQ_{k,m-1}^{A,B}.
	\end{align*}
	
	To show $(ii)$ we will use proof by induction with regard to $m\in\mathbb{N}_{0}$ via the definition of $Q_{k,m}^{A,B}$ \eqref{important}.  
	
	Obviously, for $m=0,1$, we have
	\begin{equation*}
	Q_{k,0}^{A,B}\coloneqq A^{k}, \quad Q_{k,1}^{A,B}=\sum_{l=0}^{k}A^{k-l}BQ_{l,0}^{A,B}=\sum_{l=0}^{k}A^{k-l}BA^{l}=(k+1)A^{k}B=\binom{k+1}{1}A^{k}B. 
	\end{equation*}
	
	Suppose that it is true for $m=n\in\mathbb{N}$:
	\begin{equation*}
	Q_{k,n}^{A,B}=\binom{k+n}{n}A^{k}B^{n}.
	\end{equation*}
	Let us prove it for $m=n+1$:
	\begin{align*}
	Q_{k,n+1}^{A,B}&=\sum_{l=0}^{k}A^{k-l}BQ_{l,n}^{A,B}=\sum_{l=0}^{k}A^{k-l}B\binom{l+n}{n}A^{l}B^{n}\\&=A^{k}B^{n+1}\sum_{l=0}^{k}\binom{l+n}{n}=\binom{k+n+1}{n+1}A^{k}B^{n+1}.
	\end{align*}
	The proof is complete.	
\end{proof}

According to the above theorem, a linear operator $Q_{k,m}^{A,B}$ for $k,m\in\mathbb{N}$ satisfies the following Pascal's rule for permutable linear operators $A,B\in\mathscr{B}(Y)$ as follows:
\begin{align}
\binom{k+m}{m}A^{k}B^{m}=\binom{k+m-1}{m}A^{k-1}B^{m}+\binom{k+m-1}{m-1}A^{k}B^{m-1}, \quad k,m\in\mathbb{N}.
\end{align}
By using the property of $Q_{k,m}^{A,B}$ \eqref{commutative} we define the following bivariate Mittag-Leffler function via permutable linear bounded operators which is similar to \eqref{bivtype}.
\begin{definition}
	We define a Mittag-Leffler  function $ E_{\alpha,\beta,\gamma}(A(\cdot)^{\alpha},B(\cdot)^{\beta}) :\mathbb{R}\to Y$ generated by permutable linear bounded operators $A,B\in\mathscr{B}(Y)$ as follows:
	\begin{equation}\label{per}
	E_{\alpha,\beta,\gamma}(At^{\alpha},Bt^{\beta})\coloneqq\sum_{k=0}^{\infty}\sum_{m=0}^{\infty}\binom{k+m}{m}A^{k}B^{m}\frac{t^{k\alpha+m\beta}}{\Gamma(k\alpha+m\beta+\gamma)}, \quad \alpha,\beta>0, \quad \gamma \in \mathbb{R}.
	\end{equation}	
\end{definition}
In the special case, bivariate Mittag-Leffler function \eqref{per} via commutative linear bounded operators converts to the product of classical exponential functions as follows.
\begin{lemma}
	If $\alpha=\beta=\gamma=1$, then we get the double exponential function:
	\begin{equation*}
	E_{1,1,1}(At,Bt)=\exp(At)\exp(Bt)=\exp((A+B)t), \quad  t \in \mathbb{R}.
	\end{equation*}
\end{lemma}
\begin{proof}
	Applying the formula \eqref{per}, we attain
	\begin{align*}
	E_{1,1,1}(At,Bt)&=\sum_{k=0}^{\infty}\sum_{m=0}^{\infty}\binom{k+m}{m}A^{k}B^{m}\frac{t^{k+m}}{(k+m)!}\\
	&=\sum_{k=0}^{\infty}A^{k}\frac{t^{k}}{k!}\sum_{m=0}^{\infty}B^{m}\frac{t^{m}}{m!}=\exp(At)\exp(Bt)=\exp((A+B)t),\quad  t \in \mathbb{R}.
	\end{align*}
\end{proof}
	
The following lemma plays a crucial role for solving the  given Cauchy problem \eqref{mtde-1} with linear bounded operators. In general case, it holds true whenever  $\alpha>0$, $\alpha>\beta$, $\gamma \in \mathbb{R}$.
\begin{lemma}\label{Q^A,B}
	For $A,B\in\mathscr{B}(Y)$ which are satisfying $AB\neq BA$, we have:
	\begin{align}\label{eq1}
	&\mathscr{L}^{-1}\left\lbrace \frac{s^{\gamma}}{s^{(m+1)\beta}} \left[ (s^{\alpha-\beta}I-A)^{-1}B\right]^{m} (s^{\alpha-\beta}I-A)^{-1}\right\rbrace (t)\nonumber\\&=\sum_{k=0}^{\infty}\frac{Q_{k,m}^{A,B}}{\Gamma(k(\alpha-\beta)+m\alpha+\alpha-\gamma)}t^{k(\alpha-\beta)+m\alpha+\alpha-\gamma-1}, \quad m\in\mathbb{N}_{0}\coloneqq\mathbb{N}\cup \left\lbrace 0\right\rbrace.
	\end{align}

\end{lemma}
\begin{proof}
To prove, we will use a mathematical induction principle with regard to $m\in \mathbb{N}_{0}$.
Obviously, according to the relation \eqref{ad}, \eqref{eq1} is true for $m=0$, which establishes the basis for induction: 
\begin{align}\label{1-relation}
&\mathscr{L}^{-1}\left\lbrace s^{\gamma-\beta}(s^{\alpha-\beta}I-A)^{-1}\right\rbrace (t)=t^{\alpha-\gamma-1}E_{\alpha-\beta,\alpha-\gamma}^{1}(At^{\alpha-\beta})\nonumber\\&=t^{\alpha-\gamma-1}E_{\alpha-\beta,\alpha-\gamma}(At^{\alpha-\beta})=\sum_{k=0}^{\infty}A^{k}\frac{t^{k(\alpha-\beta)+\alpha-\gamma-1}}{\Gamma(k(\alpha-\beta)+\alpha-\gamma)}\nonumber\\&=\sum_{k=0}^{\infty}Q_{k,0}^{A,B}\frac{t^{k(\alpha-\beta)+\alpha-\gamma-1}}{\Gamma(k(\alpha-\beta)+\alpha-\gamma)}, \quad \text{where} \quad Q_{k,0}^{A,B}\coloneqq A^{k}, \quad k\in \mathbb{N}_{0}.
\end{align}
For $m=1$, we use the convolution property of Laplace integral transform and formula \eqref{1-relation}:
\allowdisplaybreaks
\begin{align}\label{relation-2}
&\mathscr{L}^{-1}\left\lbrace s^{\gamma-2\beta}(s^{\alpha-\beta}I-A)^{-1}B(s^{\alpha-\beta}I-A)^{-1}\right\rbrace (t)\nonumber\\=&\mathscr{L}^{-1}\left\lbrace s^{-\beta}(s^{\alpha-\beta}I-A)^{-1}B\right\rbrace (t)\ast\mathscr{L}^{-1}\left\lbrace s^{\gamma-\beta}(s^{\alpha-\beta}I-A)^{-1}\right\rbrace(t)\nonumber\\=&t^{\alpha-1}E_{\alpha-\beta,\alpha}(At^{\alpha-\beta})B\ast t^{\alpha-\gamma-1}E_{\alpha-\beta,\alpha-\gamma}(At^{\alpha-\beta})\nonumber\\=&
\int\limits_{0}^{t}(t-s)^{\alpha-1}E_{\alpha-\beta,\alpha}(A(t-s)^{\alpha-\beta})B
s^{\alpha-\gamma-1}E_{\alpha-\beta,\alpha-\gamma}(As^{\alpha-\beta})\mathrm{d}s.
\end{align}
Then interchanging the order of integration and summation in \eqref{relation-2} which is permissible in accordance with the uniform convergence of the series \eqref{ML-2}, we attain:
\allowdisplaybreaks
\begin{align}\label{relation-3}
&\mathscr{L}^{-1}\left\lbrace s^{\gamma-2\beta}(s^{\alpha-\beta}I-A)^{-1}B(s^{\alpha-\beta}I-A)^{-1}\right\rbrace (t)\nonumber\\=&\sum_{k=0}^{\infty}\sum_{l=0}^{\infty}\frac{A^{k}BA^{l}}{\Gamma(k(\alpha-\beta)+\alpha)\Gamma(l(\alpha-\beta)+\alpha-\gamma)}\int\limits_{0}^{t}(t-s)^{k(\alpha-\beta)+\alpha-1}s^{l(\alpha-\beta)+\alpha-\gamma-1}\mathrm{d}s\nonumber\\=&\sum_{k=0}^{\infty}\sum_{l=0}^{\infty}\frac{A^{k}BA^{l}}{\Gamma(k(\alpha-\beta)+\alpha)\Gamma(l(\alpha-\beta)+\alpha-\gamma)}t^{(k+l)(\alpha-\beta)+2\alpha-\gamma-1}\mathcal{B}(k(\alpha-\beta)+\alpha,l(\alpha-\beta)+\alpha-\gamma)\nonumber\\=&\sum_{k=0}^{\infty}\sum_{l=0}^{\infty}\frac{A^{k}BA^{l}}{\Gamma((k+l)(\alpha-\beta)+2\alpha-\gamma)}t^{(k+l)(\alpha-\beta)+2\alpha-\gamma-1},
\end{align}
where $\mathcal{B}(\cdot,\cdot)$ is a well-known beta function.

Applying Cauchy product formula to the double infinity series in \eqref{relation-3}, we get:
\allowdisplaybreaks
\begin{align}\label{relation-4}
&\mathscr{L}^{-1}\left\lbrace s^{\gamma-2\beta}(s^{\alpha-\beta}I-A)^{-1}B(s^{\alpha-\beta}I-A)^{-1}\right\rbrace (t)\nonumber\\=&\sum_{k=0}^{\infty}\sum_{l=0}^{k}\frac{A^{k-l}BA^{l}}{\Gamma(k(\alpha-\beta)+2\alpha-\gamma)}t^{k(\alpha-\beta)+2\alpha-\gamma-1}\nonumber\\
=&\sum_{k=0}^{\infty}\sum_{l=0}^{k}\frac{A^{k-l}BQ_{l,0}^{A,B}}{\Gamma(k(\alpha-\beta)+2\alpha-\gamma)}t^{k(\alpha-\beta)+2\alpha-\gamma-1}\nonumber\\=&\sum_{k=0}^{\infty}\frac{Q_{k,1}^{A,B}}{\Gamma(k(\alpha-\beta)+\alpha+\alpha-\gamma)}t^{k(\alpha-\beta)+\alpha+\alpha-\gamma-1}, \quad\text{where} \quad Q_{k,1}^{A,B}\coloneqq\sum_{l=0}^{k} A^{k-l}BQ_{l,0}^{A,B},\quad k\in \mathbb{N}_{0}.
\end{align}
\allowdisplaybreaks
To verify the induction step, we assume that \eqref{eq1} holds true for $m=n$ where $n\in\mathbb{N}_{0}$:

\begin{align}\label{relation-5}
&\mathscr{L}^{-1}\left\lbrace s^{\gamma-(n+1)\beta}\left[ (s^{\alpha-\beta}I-A)^{-1}B\right]^{n} (s^{\alpha-\beta}I-A)^{-1}\right\rbrace (t)\nonumber\\=&\sum_{k=0}^{\infty}\sum_{l=0}^{k}\frac{A^{k-l}BQ_{l,n-1}^{A,B}}{\Gamma(k(\alpha-\beta)+(n+1)\alpha-\gamma)}t^{k(\alpha-\beta)+(n+1)\alpha-\gamma-1}\nonumber\\=&\sum_{k=0}^{\infty}\frac{Q_{k,n}^{A,B}}{\Gamma(k(\alpha-\beta)+n\alpha+\alpha-\gamma)}t^{k(\alpha-\beta)+n\alpha+\alpha-\gamma-1}, \quad\text{where} \quad Q_{k,n}^{A,B}\coloneqq\sum_{l=0}^{k} A^{k-l}BQ_{l,n-1}^{A,B},\quad k\in \mathbb{N}_{0}.
\end{align}
Then it yields that for $m=n+1$ as follows:
\allowdisplaybreaks
\begin{align}\label{relation-6}
&\mathscr{L}^{-1}\left\lbrace s^{\gamma-(n+2)\beta}\left[ (s^{\alpha-\beta}I-A)^{-1}B\right]^{n+1} (s^{\alpha-\beta}I-A)^{-1}\right\rbrace (t)\nonumber\\
=&\mathscr{L}^{-1}\left\lbrace s^{-\beta}(s^{\alpha-\beta}I-A)^{-1}B\right\rbrace (t)\ast\mathscr{L}^{-1}\left\lbrace s^{\gamma-(n+1)\beta}\left[ (s^{\alpha-\beta}I-A)^{-1}B\right] ^{n}(s^{\alpha-\beta}I-A)^{-1}\right\rbrace(t)\nonumber\\=&t^{\alpha-1}E_{\alpha-\beta,\alpha}(At^{\alpha-\beta})B\ast \sum_{l=0}^{\infty}\frac{Q_{l,n}^{A,B}}{\Gamma(l(\alpha-\beta)+(n+1)\alpha-\gamma)}t^{l(\alpha-\beta)+(n+1)\alpha-\gamma-1}\nonumber\\=&
\int\limits_{0}^{t}(t-s)^{\alpha-1}E_{\alpha-\beta,\alpha}(A(t-s)^{\alpha-\beta})B
\sum_{l=0}^{\infty}\frac{Q_{l,n}^{A,B}}{\Gamma(l(\alpha-\beta)+(n+1)\alpha-\gamma)}s^{l(\alpha-\beta)+(n+1)\alpha-\gamma-1}\mathrm{d}s\nonumber\\
=&\sum_{k=0}^{\infty}\sum_{l=0}^{\infty}\frac{A^{k}BQ_{l,n}^{A,B}}{\Gamma(k(\alpha-\beta)+\alpha)\Gamma(l(\alpha-\beta)+(n+1)\alpha-\gamma)}\int\limits_{0}^{t}(t-s)^{k(\alpha-\beta)+\alpha-1}s^{l(\alpha-\beta)+(n+1)\alpha-\gamma-1}\mathrm{d}s\nonumber\\=&\sum_{k=0}^{\infty}\sum_{l=0}^{\infty}\frac{A^{k}BQ_{l,n}^{A,B}}{\Gamma(k(\alpha-\beta)+\alpha)\Gamma(l(\alpha-\beta)+(n+1)\alpha-\gamma)}t^{(k+l)(\alpha-\beta)+(n+1)\alpha+\alpha-\gamma-1}\nonumber\\\times&\mathcal{B}(k(\alpha-\beta)+\alpha,l(\alpha-\beta)+(n+1)\alpha-\gamma)\nonumber\\=&\sum_{k=0}^{\infty}\sum_{l=0}^{\infty}\frac{A^{k}BQ_{l,n}^{A,B}}{\Gamma((k+l)(\alpha-\beta)+(n+1)\alpha+\alpha-\gamma)}t^{(k+l)(\alpha-\beta)+(n+1)\alpha+\alpha-\gamma-1}\nonumber\\
=&\sum_{k=0}^{\infty}\sum_{l=0}^{k}\frac{A^{k-l}BQ_{l,n}^{A,B}}{\Gamma(k(\alpha-\beta)+(n+1)\alpha+\alpha-\gamma)}t^{k(\alpha-\beta)+(n+1)\alpha+\alpha-\gamma-1}\nonumber\\=&\sum_{k=0}^{\infty}\frac{Q_{k,n+1}^{A,B}}{\Gamma(k(\alpha-\beta)+(n+1)\alpha+\alpha-\gamma)}t^{k(\alpha-\beta)+(n+1)\alpha+\alpha-\gamma-1}, \quad\text{where} \quad Q_{k,n+1}^{A,B}\coloneqq\sum_{l=0}^{k} A^{k-l}BQ_{l,n}^{A,B},\quad k\in \mathbb{N}_{0}.
\end{align}
Thus, \eqref{relation-6} holds true  whenever \eqref{relation-5} is true, and by the principle of mathematical induction, we conclude that the formula \eqref{eq1} holds true for all $m \in \mathbb{N}_{0}$.
\end{proof}

\begin{theorem}	Let $A,B\in\mathscr{B}(Y)$ with non-zero commutator, i.e., $\left[ A, B\right] \coloneqq AB- BA \neq 0$. Assume that $g(\cdot): \mathbb{J} \to Y$ and $\left( \prescript{C}{}{D^{\beta}_{0+}}x\right) (t)$ where $0<\beta\leq 1$ are exponentially bounded.
A mild solution $x(\cdot)\in \mathbb{C}^{2}(\mathbb{J},Y)$ of the Cauchy problem \eqref{mtde-1} can be represented as
\allowdisplaybreaks
\begin{align}
x(t)&=\left(I+ \sum_{k=0}^{\infty}\sum_{m=0}^{\infty}Q_{k,m}^{A,B}B\frac{t^{k(\alpha-\beta)+m\alpha+\alpha}}{\Gamma(k(\alpha-\beta)+m\alpha+\alpha+1)}\right)\eta+
\sum_{k=0}^{\infty}\sum_{m=0}^{\infty}Q_{k,m}^{A,B}\frac{t^{k(\alpha-\beta)+m\alpha+1}}{\Gamma(k(\alpha-\beta)+m\alpha+2)}\hat{\eta}\nonumber\\
&+\int\limits_{0}^{t}\sum_{k=0}^{\infty}\sum_{m=0}^{\infty}Q_{k,m}^{A,B}\frac{(t-s)^{k(\alpha-\beta)+m\alpha+\alpha-1}}{\Gamma(k(\alpha-\beta)+m\alpha+\alpha)}g(s)\mathrm{d}s\nonumber\\&=\left( I+t^{\alpha}\mathscr{E}_{\alpha-\beta,\alpha,\alpha+1}^{A,B}(t)B\right) \eta+t\mathscr{E}_{\alpha-\beta,\alpha,2}^{A,B}(t)\hat{\eta}+\int\limits_{0}^{t}(t-s)^{\alpha-1}\mathscr{E}_{\alpha-\beta,\alpha,\alpha}^{A,B}(t-s)g(s)\mathrm{d}s, \quad t>0,
\end{align}
where $I \in \mathscr{B}(Y)$ is an identity operator.
\end{theorem}
\begin{proof}
	We recall that the existence of Laplace transform of $x(\cdot)$ and its Caputo derivatives $\prescript{C}{}{D^{\alpha}_{0^{+}}x(\cdot)}$ and $\prescript{C}{}{D^{\beta}_{0^{+}}x(\cdot)}$  for $ 1<\alpha\leq 2$ and $0<\beta\leq 1$, respectively, is guaranteed by Theorem \ref{thm2}.
	Thus, to find the mild solution $x(t)$ of \eqref{mtde-1} satisfying the  initial conditions $x(0)=\eta$, $x'(0)=\hat{\eta}$, we can use the Laplace integral transform. By assuming $T=\infty$, taking the Laplace transform on both sides of equation \eqref{mtde-1} and using the following facts that 
	\begin{align*}
	&\mathscr{L}\left\lbrace \prescript{C}{}{D^{\alpha}_{0^{+}}x(t)} \right\rbrace(s)=s^{\alpha}X(s)-s^{\alpha-1}\eta-s^{\alpha-2}\hat{\eta}, \\
  &\mathscr{L}\left\lbrace \prescript{C}{}{D^{\beta}_{0^{+}}x(t)} \right\rbrace(s)=s^{\beta}X(s)-s^{\beta-1}\eta,
	\end{align*}
	which implies that 
\begin{align*}
	\left( s^{\alpha} I-As^{\beta}-B\right) X(s)=s^{\alpha-1}\eta +s^{\alpha-2}\hat{\eta}-s^{\beta-1}A \eta +G(s),
\end{align*}
where $X(s)$ and $G(s)$ represent the Laplace integral transforms of $x(t)$ and $g(t)$, respectively. 

Thus, after solving the above equation with respect to the $X(s)$, we get 
\begin{align*}
	X(s)&=s^{\alpha-1}\left( s^{\alpha}I-As^{\beta}-B\right)^{-1}\eta+s^{\alpha-2} \left( s^{\alpha}I-As^{\beta}-B\right)^{-1}\hat{\eta}\\
	&-s^{\beta-1}\left( s^{\alpha}I-As^{\beta}-B\right)^{-1}A \eta +\left( s^{\alpha-1}I-As^{\beta}-B\right)^{-1}G(s)\\
	&=s^{-1}\eta+s^{-1}\left( s^{\alpha}I-As^{\beta}-B\right)^{-1}B\eta+s^{\alpha-2} \left( s^{\alpha}I-As^{\beta}-B\right)^{-1}\hat{\eta}\\&+\left( s^{\alpha-1}I-As^{\beta}-B\right)^{-1}G(s).
\end{align*}

On the other hand, in accordance with the relation \eqref{operator}, for sufficiently large $s$, such that

\begin{equation*}
\|(s^{\alpha-\beta}I-A)^{-1}Bs^{-\beta}\|< 1.
\end{equation*}

Thus, for nonpermutable linear operators $A,B\in\mathscr{B}(Y)$ and sufficiently large $s$, we have
\begin{align*}
	\left( s^{\alpha}I-As^{\beta}-B\right)^{-1}&=\left(s^{\beta}\left[  s^{\alpha-\beta}I-A-Bs^{-\beta}\right] \right)^{-1}\\
	&=\left(s^{\beta}(s^{\alpha-\beta}I-A)\left[ I-(s^{\alpha-\beta}I-A)^{-1}Bs^{-\beta}\right] \right) ^{-1}\\
	&=\left(s^{\beta}\left[ I-\left( s^{\alpha-\beta}I-A\right) ^{-1}Bs^{-\beta}\right]\right)^{-1}\left(s^{\alpha-\beta}I-A \right)^{-1} \\
	&=\left[I-(s^{\alpha-\beta}I-A)^{-1}Bs^{-\beta} \right]^{-1}s^{-\beta}\left(s^{\alpha-\beta}I-A \right)^{-1} \\
	&=\sum_{m=0}^{\infty}\frac{1}{s^{\beta m}}\left[ \left(s^{\alpha-\beta}I-A \right)^{-1}B \right]^{m}s^{-\beta}\left(s^{\alpha-\beta}I-A \right)^{-1}\\
	&=\sum_{m=0}^{\infty}\frac{1}{s^{(m+1)\beta}}\left[ \left(s^{\alpha-\beta}I-A \right)^{-1}B \right]^{m}\left(s^{\alpha-\beta}I-A \right)^{-1}
\end{align*}
Then, by taking inverse Laplace transform, we have
\allowdisplaybreaks
\begin{align} \label{x(t)}
	x(t)&=\mathscr{L}^{-1}\left\lbrace s^{-1}\right\rbrace(t) \eta+ \mathscr{L}^{-1}\left\lbrace\sum_{m=0}^{\infty}\frac{s^{-1}}{s^{ (m+1)\beta}}\left[ \left(s^{\alpha-\beta}I-A \right)^{-1}B \right]^{m}\left(s^{\alpha-\beta}I-A \right)^{-1}\right\rbrace (t)B\eta \nonumber\\
	&+\mathscr{L}^{-1}\left\lbrace\sum_{m=0}^{\infty} \frac{s^{\alpha-2}}{s^{(m+1)\beta}}\left[ \left(s^{\alpha-\beta}I-A \right)^{-1}B \right]^{m}\left(s^{\alpha-\beta}I-A \right)^{-1}\right\rbrace (t)\hat{\eta}\nonumber\\
	&+\mathscr{L}^{-1}\left\lbrace \sum_{m=0}^{\infty}\frac{1}{s^{(m+1)\beta}}\left[ \left(s^{\alpha-\beta}I-A \right)^{-1}B \right]^{m}\left(s^{\alpha-\beta}I-A \right)^{-1} G(s)\right\rbrace (t).
\end{align}
Therefore, in accordance with Lemma \ref{Q^A,B}, we acquire
\allowdisplaybreaks
\begin{align}
x(t)&=\left(I+ \sum_{k=0}^{\infty}\sum_{m=0}^{\infty}Q_{k,m}^{A,B}B\frac{t^{k(\alpha-\beta)+m\alpha+\alpha}}{\Gamma(k(\alpha-\beta)+m\alpha+\alpha+1)}\right)\eta+
\sum_{k=0}^{\infty}\sum_{m=0}^{\infty}Q_{k,m}^{A,B}\frac{t^{k(\alpha-\beta)+m\alpha+1}}{\Gamma(k(\alpha-\beta)+m\alpha+2)}\hat{\eta}\nonumber\\
&+\int\limits_{0}^{t}\sum_{k=0}^{\infty}\sum_{m=0}^{\infty}Q_{k,m}^{A,B}\frac{(t-s)^{k(\alpha-\beta)+m\alpha+\alpha-1}}{\Gamma(k(\alpha-\beta)+m\alpha+\alpha)}g(s)\mathrm{d}s\nonumber\\
&\coloneqq\left( I+t^{\alpha}\mathscr{E}_{\alpha-\beta,\alpha,\alpha+1}^{A,B}(t)B\right) \eta+t\mathscr{E}_{\alpha-\beta,\alpha,2}^{A,B}(t)\hat{\eta}+\int\limits_{0}^{t}(t-s)^{\alpha-1}\mathscr{E}_{\alpha-\beta,\alpha,\alpha}^{A,B}(t-s)g(s)\mathrm{d}s,\quad t>0.
\end{align}
\end{proof}

It should stressed out that the assumption on the exponential boundedness of the function $g(\cdot)$ and $\left( \prescript{C}{}{D^{\beta}_{0+}}x\right) (\cdot)$ where $0<\beta\leq 1$ (alternatively, $\left( \prescript{C}{}{D^{\alpha}_{0+}}x\right) (\cdot)$ for $1<\alpha\leq2$)  can be omitted. As is shown below, the statement of the above theorem holds for a more general function $g(\cdot)\in \mathbb{C}(\mathbb{J},Y)$.
\begin{theorem}
		Let $A,B\in\mathscr{B}(Y)$ with non-zero commutator, i.e., $\left[ A, B\right] \coloneqq AB- BA \neq 0$. 
	A mild solution $x(\cdot)\in \mathbb{C}^{2}(\mathbb{J},Y)$ of the Cauchy problem \eqref{mtde-1} can be represented as
	\allowdisplaybreaks
	\begin{align}
	x(t)&=\left(I+ \sum_{k=0}^{\infty}\sum_{m=0}^{\infty}Q_{k,m}^{A,B}B\frac{t^{k(\alpha-\beta)+m\alpha+\alpha}}{\Gamma(k(\alpha-\beta)+m\alpha+\alpha+1)}\right)\eta+
	\sum_{k=0}^{\infty}\sum_{m=0}^{\infty}Q_{k,m}^{A,B}\frac{t^{k(\alpha-\beta)+m\alpha+1}}{\Gamma(k(\alpha-\beta)+m\alpha+2)}\hat{\eta}\nonumber\\
	&+\int\limits_{0}^{t}\sum_{k=0}^{\infty}\sum_{m=0}^{\infty}Q_{k,m}^{A,B}\frac{(t-s)^{k(\alpha-\beta)+m\alpha+\alpha-1}}{\Gamma(k(\alpha-\beta)+m\alpha+\alpha)}g(s)\mathrm{d}s\nonumber\\
	&\coloneqq\left( I+t^{\alpha}\mathscr{E}_{\alpha-\beta,\alpha,\alpha+1}^{A,B}(t)B\right) \eta+t\mathscr{E}_{\alpha-\beta,\alpha,2}^{A,B}(t)\hat{\eta}+\int\limits_{0}^{t}(t-s)^{\alpha-1}\mathscr{E}_{\alpha-\beta,\alpha,\alpha}^{A,B}(t-s)g(s)\mathrm{d}s,\quad t>0.
	\end{align}
\end{theorem}

\begin{proof}
	For making use of verification by substitution, we apply superposition principle for the initial value problem of linear inhomogeneous multi-order fractional evolution equation \eqref{mtde-1}. 
   For this, firstly let us consider the following homogeneous system with inhomogeneous initial conditions:
	
	\begin{equation} \label{eq:f5hom}
	\begin{cases}
	\left( \prescript{C}{}{D^{\alpha}_{0+}}x\right) (t)-A \left( \prescript{C}{}{D^{\beta}_{0+}}x\right) (t)-Bx(t)=0,\quad t>0\\
	x(0)=\eta , \quad x^{\prime}(0)=\tilde{\eta},
	\end{cases}
	\end{equation}
	has a mild solution 
	\begin{align}\label{x(t)-hom}
	x(t)=&\left(I+ \sum_{k=0}^{\infty}\sum_{m=0}^{\infty}Q_{k,m}^{A,B}B\frac{t^{k(\alpha-\beta)+m\alpha+\alpha}}{\Gamma(k(\alpha-\beta)+m\alpha+\alpha+1)}\right)\eta+
	\sum_{k=0}^{\infty}\sum_{m=0}^{\infty}Q_{k,m}^{A,B}\frac{t^{k(\alpha-\beta)+m\alpha+1}}{\Gamma(k(\alpha-\beta)+m\alpha+2)}\hat{\eta}\nonumber\nonumber\\=&\left( I+t^{\alpha}\mathscr{E}^{A,B}_{\alpha-\beta, \alpha,\alpha+1}(t)B\right) \eta+t\mathscr{E}^{A,B}_{\alpha-\beta, \alpha,2}(t)\tilde{\eta}.
	\end{align}
	With the help of verification by substitution and the property of $Q^{A,B}_{k,m}$ \eqref{gen-Pascal}, we confirm that  \eqref{x(t)-hom} is a mild solution of linear homogeneous fractional evolution equation \eqref{eq:f5hom}: 
	\allowdisplaybreaks
	\begin{align*}
	\left( \prescript{C}{}{D^{\alpha}_{0+}}x\right) (t)&= \prescript{C}{}{D^{\alpha}_{0+}}\left( I+ t^{\alpha}\mathscr{E}^{A,B}_{\alpha-\beta, \alpha,\alpha+1}(t)B\right) \eta+ \prescript{C}{}{D^{\alpha}_{0+}} \left( t\mathscr{E}^{A,B}_{\alpha-\beta, \alpha,2}(t)\right) \hat{\eta}\\
	&=\prescript{C}{}{D^{\alpha}_{0+}}\left(I+ \sum_{k=0}^{\infty}\sum_{m=0}^{\infty}Q^{A,B}_{k,m} B\frac{t^{k(\alpha-\beta)+m\alpha+\alpha}}{\Gamma(k(\alpha-\beta)+m\alpha+\alpha+1)}\right) \eta\\
	&+\prescript{C}{}{D^{\alpha}_{0+}}\left( \sum_{k=0}^{\infty}\sum_{m=0}^{\infty}Q^{A,B}_{k,m} \frac{t^{k(\alpha-\beta)+m\alpha+1}}{\Gamma(k(\alpha-\beta)+m\alpha+2)}\right) \hat{\eta}.
	\end{align*}
	In this case, we first apply the property of $Q^{A,B}_{k,m}$ \eqref{gen-Pascal} before Caputo differentiation the first and second terms above, in accordance with the following formula \cite{Podlubny}:
	\begin{equation}\label{formula}
	\prescript{C}{}D^{\nu}_{0+}\left(\frac{t^{\eta}}{\Gamma(\eta+1)}\right)= 	\begin{cases}
	\frac{t^{\eta-\nu}}{\Gamma(\eta-\nu+1)},\quad\eta>\lfloor\nu\rfloor,\\
	0, \qquad \qquad  \eta=0,1,2,\ldots, \lfloor\nu\rfloor,\\
	\text{undefined},  \qquad \text{otherwise}.
	\end{cases}
	\end{equation}
	Then, we have 
	\begin{align*}
	\left(\prescript{C}{}{D^{\alpha}_{0+}}x\right)(t)&=\prescript{C}{}{D^{\alpha}_{0+}}\Big[B\frac{t^{\alpha}}{\Gamma(\alpha)}+ \sum_{k=1}^{\infty}\sum_{m=0}^{\infty}AQ^{A,B}_{k-1,m} B\frac{t^{k(\alpha-\beta)+m\alpha+\alpha}}{\Gamma(k(\alpha-\beta)+m\alpha+\alpha+1)}\\&+\sum_{k=0}^{\infty}\sum_{m=1}^{\infty}BQ^{A,B}_{k,m-1} B\frac{t^{k(\alpha-\beta)+m\alpha+\alpha}}{\Gamma(k(\alpha-\beta)+m\alpha+\alpha+1)} \Big] \eta\\
	&+\prescript{C}{}{D^{\alpha}_{0+}}\Big[tI+\sum_{k=1}^{\infty}\sum_{m=0}^{\infty}AQ^{A,B}_{k-1,m} \frac{t^{k(\alpha-\beta)+m\alpha+1}}{\Gamma(k(\alpha-\beta)+m\alpha+2)}\\
	&+\sum_{k=0}^{\infty}\sum_{m=1}^{\infty}BQ^{A,B}_{k,m-1} \frac{t^{k(\alpha-\beta)+m\alpha+1}}{\Gamma(k(\alpha-\beta)+m\alpha+2)}\Big] \hat{\eta}\\
	&=B\eta+\sum_{k=1}^{\infty}\sum_{m=0}^{\infty}AQ^{A,B}_{k-1,m} B\frac{t^{k(\alpha-\beta)+m\alpha}}{\Gamma(k(\alpha-\beta)+m\alpha+1)}\eta\\
	&+\sum_{k=0}^{\infty}\sum_{m=1}^{\infty}BQ^{A,B}_{k,m-1} B\frac{t^{k(\alpha-\beta)+m\alpha}}{\Gamma(k(\alpha-\beta)+m\alpha+1)}\eta\\
	&+\sum_{k=1}^{\infty}\sum_{m=0}^{\infty}AQ^{A,B}_{k-1,m} \frac{t^{k(\alpha-\beta)+m\alpha+1-\alpha}}{\Gamma(k(\alpha-\beta)+m\alpha+2-\alpha)}\hat{\eta}\\
	&+\sum_{k=0}^{\infty}\sum_{m=1}^{\infty}BQ^{A,B}_{k,m-1} \frac{t^{k(\alpha-\beta)+m\alpha+1-\alpha}}{\Gamma(k(\alpha-\beta)+m\alpha+2-\alpha)} \hat{\eta}.
	\end{align*}
	Next, we can attain that
	\begin{align*}
	\left(\prescript{C}{}{D^{\alpha}_{0+}}x\right)(t)
	&=B\eta+\sum_{k=0}^{\infty}\sum_{m=0}^{\infty}AQ^{A,B}_{k,m} B\frac{t^{k(\alpha-\beta)+m\alpha+\alpha-\beta}}{\Gamma(k(\alpha-\beta)+m\alpha+\alpha-\beta+1)}\eta\\
	&+\sum_{k=0}^{\infty}\sum_{m=0}^{\infty}BQ^{A,B}_{k,m} B\frac{t^{k(\alpha-\beta)+m\alpha+\alpha}}{\Gamma(k(\alpha-\beta)+m\alpha+\alpha+1)}\eta\\
	&+\sum_{k=0}^{\infty}\sum_{m=0}^{\infty}AQ^{A,B}_{k,m} \frac{t^{k(\alpha-\beta)+m\alpha+1-\beta}}{\Gamma(k(\alpha-\beta)+m\alpha+2-\beta)}\hat{\eta}\\
	&+\sum_{k=0}^{\infty}\sum_{m=0}^{\infty}BQ^{A,B}_{k,m} \frac{t^{k(\alpha-\beta)+m\alpha+1}}{\Gamma(k(\alpha-\beta)+m\alpha+2)} \hat{\eta}.
	\end{align*}
	Then, the Caputo fractional differentiation of $x(t)$ \eqref{x(t)-hom} of order $0<\beta\leq 1$ is as follows:
	
	\begin{align*}
	\left( \prescript{C}{}{D^{\beta}_{0+}}x\right) (t)&= \prescript{C}{}{D^{\beta}_{0+}}\left( I+t^{\alpha} \mathscr{E}^{A,B}_{\alpha-\beta, \alpha,\alpha+1}(t)\right) \eta+ \prescript{C}{}{D^{\beta}_{0+}}\left( t \mathscr{E}^{A,B}_{\alpha-\beta, \alpha,2}(t)\right) \hat{\eta}\\
	&= \prescript{C}{}{D^{\beta}_{0+}}\left(I+ \sum_{k=0}^{\infty}\sum_{m=0}^{\infty}Q^{A,B}_{k,m} B\frac{t^{k(\alpha-\beta)+m\beta+\alpha}}{\Gamma(k(\alpha-\beta)+m\beta+\alpha+1)}\right) \eta\\
	&+\prescript{C}{}{D^{\beta}_{0+}}\left( \sum_{k=0}^{\infty}\sum_{m=0}^{\infty}Q^{A,B}_{k,m} \frac{t^{k(\alpha-\beta)+m\beta+1}}{\Gamma(k(\alpha-\beta)+m\beta+2)}\right) \hat{\eta}\\
	&=\sum_{k=0}^{\infty}\sum_{m=0}^{\infty}Q^{A,B}_{k,m} B\frac{t^{k(\alpha-\beta)+m\beta+\alpha-\beta}}{\Gamma(k(\alpha-\beta)+m\beta+\alpha-\beta+1)}\eta\\
	&+\sum_{k=0}^{\infty}\sum_{m=0}^{\infty}Q^{A,B}_{k,m} \frac{t^{k(\alpha-\beta)+m\beta+1-\beta}}{\Gamma(k(\alpha-\beta)+m\beta+2-\beta)} \hat{\eta}.
	\end{align*}
	Finally, taking a linear combination of above results, we acquire the desired result:
	\begin{align*}
	\left(\prescript{C}{}{D^{\alpha}_{0+}}x\right)(t)-A\left( \prescript{C}{}{D^{\beta}_{0+}}x\right)(t)&=B\eta+\sum_{k=0}^{\infty}\sum_{m=0}^{\infty}BQ^{A,B}_{k,m} B\frac{t^{k(\alpha-\beta)+m\alpha+\alpha}}{\Gamma(k(\alpha-\beta)+m\alpha+\alpha+1)}\eta\\&+\sum_{k=0}^{\infty}\sum_{m=0}^{\infty}BQ^{A,B}_{k,m} \frac{t^{k(\alpha-\beta)+m\alpha+1}}{\Gamma(k(\alpha-\beta)+m\alpha+2)}\hat{\eta}\coloneqq Bx(t).
	\end{align*}
	
	Next, we consider the following linear inhomogeneous fractional evolution equation:
	\begin{equation} \label{eq:f5inhom}
	\left( \prescript{C}{}{D^{\alpha}_{0+}}x\right) (t)-A \left( \prescript{C}{}{D^{\beta}_{0+}}x\right) (t)-Bx(t)=g(t),
	\end{equation}
	with zero initial conditions:
	\begin{equation*}
	x(0)=x'(0)=0,
	\end{equation*}
	 has an integral representation of a mild solution which is a particular solution of \eqref{mtde-1}:
	\begin{align*}
	\bar{x}(t)=\int\limits_{0}^{t}(t-s)^{\alpha-1}\mathscr{E}^{A,B}_{\alpha-\beta, \alpha,\alpha}(t-s)g(s)\mathrm{d}s=\int\limits_{0}^{t}\sum_{k=0}^{\infty}\sum_{m=0}^{\infty}Q^{A,B}_{k,m}\frac{(t-s)^{k(\alpha-\beta)+m\alpha+\alpha-1}}{\Gamma(k(\alpha-\beta)+m\alpha+\alpha)}g(s)\mathrm{d}s,\quad t>0.
	\end{align*}
	In accordance with fractional analogue of variation of constants formula any particular mild solution of inhomogeneous differential equation of fractional-order \eqref{eq:f5inhom} should be looked for in the form of
	
	\begin{equation}
	\bar{x}(t)=\int\limits_{0}^{t}(t-s)^{\alpha-1}\mathscr{E}^{A,B}_{\alpha-\beta, \alpha,\alpha}(t-s)f(s)\mathrm{d}s=\int\limits_{0}^{t}\sum_{k=0}^{\infty}\sum_{m=0}^{\infty}Q^{A,B}_{k,m}\frac{(t-s)^{k(\alpha-\beta)+m\alpha+\alpha-1}}{\Gamma(k(\alpha-\beta)+m\alpha+\alpha)}f(s)\mathrm{d}s, \quad t>0,
	\end{equation}
	where $f(s)$ is unknown function for $s\in [0,t]$ with $\bar{x}(0)=\bar{x}'(0)=0$.
	
	 Because of this homogeneous initial values $\bar{x}(0)=\bar{x}'(0)=0$, it follows that in this case, for any given order either in $(1, 2]$ and $(0, 1]$, the Riemann–Liouville and Caputo type fractional differentiation operators are equal in accordance with \eqref{relation}. Therefore, in the work below we will apply Riemann–Liouville derivative instead of Caputo one to verify the mild solution of evolution equation with two independent fractional-orders.

	Applying the property of a linear operator $Q_{k,m}^{A,B}$ \eqref{gen-Pascal} and having Caputo differentiation of order \\ $1<\alpha\leq 2$ of $\bar{x}(t)$, we obtain:
	\allowdisplaybreaks
	\begin{align*}
	&\left( \prescript{C}{}{D^{\alpha}_{0+}}\bar{x}\right) (t)=\left( \prescript{RL}{}{D^{\alpha}_{0+}}\bar{x}\right) (t)\\&= \prescript{RL}{}{D^{\alpha}_{0+}}\Big[\int\limits_{0}^{t}\frac{(t-s)^{\alpha-1}}{\Gamma(\alpha)}f(s)\mathrm{d}s+\int\limits_{0}^{t}\sum_{k=1}^{\infty}\sum_{m=0}^{\infty}AQ^{A,B}_{k-1,m}\frac{(t-s)^{k(\alpha-\beta)+m\alpha+\alpha-1}}{\Gamma(k(\alpha-\beta)+m\alpha+\alpha)}f(s)\mathrm{d}s\\
	&+\int\limits_{0}^{t}\sum_{k=0}^{\infty}\sum_{m=1}^{\infty}BQ^{A,B}_{k,m-1}\frac{(t-s)^{k(\alpha-\beta)+m\alpha+\alpha-1}}{\Gamma(k(\alpha-\beta)+m\alpha+\alpha)}f(s)\mathrm{d}s\Big]\\
	&=\left( \prescript{RL}{}{D^{\alpha}_{0+}}(\prescript{}{}{I^{\alpha}_{0+}}f)\right)(t)+\prescript{RL}{}{D^{\alpha}_{0+}}\Big[\int\limits_{0}^{t}\sum_{k=0}^{\infty}\sum_{m=0}^{\infty}AQ^{A,B}_{k,m}\frac{(t-s)^{k(\alpha-\beta)+m\alpha+2\alpha-\beta-1}}{\Gamma(k(\alpha-\beta)+m\alpha+2\alpha-\beta)}f(s)\mathrm{d}s\Big]\\
	&+\prescript{RL}{}{D^{\alpha}_{0+}}\Big[\int\limits_{0}^{t}\sum_{k=0}^{\infty}\sum_{m=0}^{\infty}BQ^{A,B}_{k,m}\frac{(t-s)^{k(\alpha-\beta)+m\alpha+2\alpha-1}}{\Gamma(k(\alpha-\beta)+m\alpha+2\alpha)}f(s)\mathrm{d}s\Big].
	\end{align*}

	By making use of the fractional Leibniz integral rules \eqref{Leibniz} in Riemann-Liouville's sense for the second and third terms of the above expression, we get
	\allowdisplaybreaks
	\begin{align*}
	&\left( \prescript{C}{}{D^{\alpha}_{0+}}\bar{x}\right) (t)=\left( \prescript{RL}{}{D^{\alpha}_{0+}}\bar{x}\right) (t)\\
	&=f(t)+\lim\limits_{s\to t-0}\prescript{RL,t}{}{D}^{\alpha-1}_{0+} \left( \sum_{k=0}^{\infty}\sum_{m=0}^{\infty}AQ^{A,B}_{k,m}\lim\limits_{s\to
	t-0}\frac{(t-s)^{k(\alpha-\beta)+m\alpha+2\alpha-\beta-1}}{\Gamma(k(\alpha-\beta)+m\alpha+2\alpha-\beta)}\right) \lim\limits_{s\to t-0}f(s)\\
	&+\lim\limits_{s\to t-0}\prescript{RL,t}{}{D}^{\alpha-2}_{0+}\left( \sum_{k=0}^{\infty}\sum_{m=0}^{\infty}AQ^{A,B}_{k,m}\lim\limits_{s\to
		t-0}\frac{(t-s)^{k(\alpha-\beta)+m\alpha+2\alpha-\beta-1}}{\Gamma(k(\alpha-\beta)+m\alpha+2\alpha-\beta)}\right)\frac{d}{dt}\lim\limits_{s\to t-0}f(s)\\
	&+\int\limits_{0}^{t}\prescript{RL,t}{}{D}^{\alpha}_{0+}\sum_{k=0}^{\infty}\sum_{m=0}^{\infty}AQ^{A,B}_{k,m}\frac{(t-s)^{k(\alpha-\beta)+m\alpha+2\alpha-\beta-1}}{\Gamma(k(\alpha-\beta)+m\alpha+2\alpha-\beta)}f(s)\mathrm{d}s\\
	&+\lim\limits_{s\to t-0}\prescript{RL,t}{}{D}^{\alpha-1}_{0+} \left(\sum_{k=0}^{\infty}\sum_{m=0}^{\infty}AQ^{A,B}_{k,m}\lim\limits_{s\to t-0}\frac{(t-s)^{k(\alpha-\beta)+m\alpha+2\alpha-1}}{\Gamma(k(\alpha-\beta)+m\alpha+2\alpha)}\right) \lim\limits_{s\to t-0}f(s)\\
	&+\lim\limits_{s\to t-0}\prescript{RL,t}{}{D}^{\alpha-2}_{0+}\left( \sum_{k=0}^{\infty}\sum_{m=0}^{\infty}AQ^{A,B}_{k,m}\lim\limits_{s\to
	t-0}\frac{(t-s)^{k(\alpha-\beta)+m\alpha+2\alpha-1}}{\Gamma(k(\alpha-\beta)+m\alpha+2\alpha)}\right)\frac{d}{dt}\lim\limits_{s\to t-0}f(s)\\
	&+\int\limits_{0}^{t}\prescript{RL,t}{}{D}^{\alpha}_{0+}\sum_{k=0}^{\infty}\sum_{m=0}^{\infty}AQ^{A,B}_{k,m}\frac{(t-s)^{k(\alpha-\beta)+m\alpha+2\alpha-1}}{\Gamma(k(\alpha-\beta)+m\alpha+2\alpha)}f(s)\mathrm{d}s\\
	&=f(t)+\lim\limits_{s\to t-0} \sum_{k=0}^{\infty}\sum_{m=0}^{\infty}AQ^{A,B}_{k,m}\lim\limits_{s\to
		t-0}\frac{(t-s)^{k(\alpha-\beta)+m\alpha+\alpha-\beta}}{\Gamma(k(\alpha-\beta)+m\alpha+\alpha-\beta+1)}\lim\limits_{s\to t-0}f(s)\\
	&+\lim\limits_{s\to t-0} \sum_{k=0}^{\infty}\sum_{m=0}^{\infty}AQ^{A,B}_{k,m}\lim\limits_{s\to
	t-0}\frac{(t-s)^{k(\alpha-\beta)+m\alpha+\alpha-\beta+1}}{\Gamma(k(\alpha-\beta)+m\alpha+\alpha-\beta+2)}\frac{d}{dt}\lim\limits_{s\to t-0}f(s)\\
	&+\int\limits_{0}^{t}\sum_{k=0}^{\infty}\sum_{m=0}^{\infty}AQ^{A,B}_{k,m}\frac{(t-s)^{k(\alpha-\beta)+m\alpha+\alpha-\beta-1}}{\Gamma(k(\alpha-\beta)+m\alpha+\alpha-\beta)}f(s)\mathrm{d}s\\
	&+\lim\limits_{s\to t-0}\prescript{RL,t}{}{D}^{\alpha-1}_{0+} \left(\sum_{k=0}^{\infty}\sum_{m=0}^{\infty}AQ^{A,B}_{k,m}\lim\limits_{s\to t-0}\frac{(t-s)^{k(\alpha-\beta)+m\alpha+\alpha}}{\Gamma(k(\alpha-\beta)+m\alpha+\alpha+1)}\right) \lim\limits_{s\to t-0}f(s)\\
	&+\lim\limits_{s\to t-0} \sum_{k=0}^{\infty}\sum_{m=0}^{\infty}AQ^{A,B}_{k,m}\lim\limits_{s\to
	t-0}\frac{(t-s)^{k(\alpha-\beta)+m\alpha+\alpha+1}}{\Gamma(k(\alpha-\beta)+m\alpha+\alpha+2)}\frac{d}{dt}\lim\limits_{s\to t-0}f(s)\\
	&+\int\limits_{0}^{t}\sum_{k=0}^{\infty}\sum_{m=0}^{\infty}AQ^{A,B}_{k,m}\frac{(t-s)^{k(\alpha-\beta)+m\alpha+\alpha-1}}{\Gamma(k(\alpha-\beta)+m\alpha+\alpha)}f(s)\mathrm{d}s\\
	&=f(t)+\int_{0}^{t}\sum_{k=0}^{\infty}\sum_{m=0}^{\infty}AQ^{A,B}_{k,m}\frac{(t-s)^{k(\alpha-\beta)+m\alpha+\alpha-\beta-1}}{\Gamma(k(\alpha-\beta)+m\alpha+\alpha-\beta)}f(s)\mathrm{d}s\\
	&+\int_{0}^{t}\sum_{k=0}^{\infty}\sum_{m=0}^{\infty}BQ^{A,B}_{k,m}\frac{(t-s)^{k(\alpha-\beta)+m\alpha+\alpha-1}}{\Gamma(k(\alpha-\beta)+m\alpha+\alpha)}f(s)\mathrm{d}s
	\end{align*}
	Then, Caputo fractional derivative of $\bar{x}(t)$ of order of $0<\beta\leq 1$ is
	\allowdisplaybreaks
	\begin{align*}
	&\left( \prescript{C}{}{D^{\beta}_{0+}}\bar{x}\right)(t)=\left( \prescript{RL}{}{D^{\beta}_{0+}}\bar{x}\right) (t)\\&=\prescript{RL}{}{D^{\beta}_{0+}}\Big[\int\limits_{0}^{t}\sum_{k=0}^{\infty}\sum_{m=0}^{\infty}Q^{A,B}_{k,m}\frac{(t-s)^{k(\alpha-\beta)+m\alpha+\alpha-1}}{\Gamma(k(\alpha-\beta)+m\alpha+\alpha)}f(s)\mathrm{d}s\Big]\\
	&=\lim\limits_{s\to t-0}\prescript{RL,t}{}{D}^{\beta-1}_{0+}\left( \sum_{k=0}^{\infty}\sum_{m=0}^{\infty}Q^{A,B}_{k,m}\frac{(t-s)^{k(\alpha-\beta)+m\alpha+\alpha-1}}{\Gamma(k(\alpha-\beta)+m\alpha+\alpha)}\right) \lim\limits_{s\to t-0}f(s)\\
	&+\int\limits_{0}^{t}\prescript{RL}{}{D^{\beta}_{0+}}\sum_{k=0}^{\infty}\sum_{m=0}^{\infty}Q^{A,B}_{k,m}\frac{(t-s)^{k(\alpha-\beta)+m\alpha+\alpha-1}}{\Gamma(k(\alpha-\beta)+m\alpha+\alpha)}f(s)\mathrm{d}s\\
	&=\lim\limits_{s\to t-0} \sum_{k=0}^{\infty}\sum_{m=0}^{\infty}Q^{A,B}_{k,m}\frac{(t-s)^{k(\alpha-\beta)+m\alpha+\alpha-\beta}}{\Gamma(k(\alpha-\beta)+m\alpha+\alpha-\beta+1)} \lim\limits_{s\to t-0}f(s)\\
	&+\int\limits_{0}^{t}\sum_{k=0}^{\infty}\sum_{m=0}^{\infty}Q^{A,B}_{k,m}\frac{(t-s)^{k(\alpha-\beta)+m\alpha+\alpha-\beta-1}}{\Gamma(k(\alpha-\beta)+m\alpha+\alpha-\beta)}f(s)\mathrm{d}s\\
	&=\int\limits_{0}^{t}\sum_{k=0}^{\infty}\sum_{m=0}^{\infty}Q^{A,B}_{k,m}\frac{(t-s)^{k(\alpha-\beta)+m\alpha+\alpha-\beta-1}}{\Gamma(k(\alpha-\beta)+m\alpha+\alpha-\beta)}f(s)\mathrm{d}s.
	\end{align*}
	Thus, linear combinations of above results yield that
	\begin{align*}
	&\left( \prescript{C}{}{D^{\alpha}_{0+}}\bar{x}\right) (t)-A \left( \prescript{C}{}{D^{\beta}_{0+}}\bar{x}\right) (t)\\&=f(t)+\int\limits_{0}^{t}\sum_{k=0}^{\infty}\sum_{m=0}^{\infty}BQ^{A,B}_{k,m}\frac{(t-s)^{k(\alpha-\beta)+m\alpha+\alpha-1}}{\Gamma(k(\alpha-\beta)+m\alpha+\alpha)}f(s)\mathrm{d}s\\&=f(t)+B\bar{x}(t)=g(t)+B\bar{x}(t).
	\end{align*}
	Therefore, $f(t)=g(t)$, $t>0$ which confirms the desired verification. The proof is complete.
\end{proof}
Then it follows that by using the substitution $y(t)=E^{-1}x(t)$, we can acquire a mild solution of \eqref{mtde} as below.
\begin{theorem}
	Let $A,B\in\mathscr{B}(Y)$ with non-zero commutator, i.e., $\left[ A, B\right] \coloneqq AB- BA \neq 0$. 
	A mild solution $y(\cdot)\in \mathbb{C}^{2}(\mathbb{J},X)$ of the Cauchy problem \eqref{mtde} can be represented as
	\allowdisplaybreaks
	\begin{align}
	y(t)&=\left(E^{-1}+ \sum_{k=0}^{\infty}\sum_{m=0}^{\infty}E^{-1}Q_{k,m}^{A,B}B\frac{t^{k(\alpha-\beta)+m\alpha+\alpha}}{\Gamma(k(\alpha-\beta)+m\alpha+\alpha+1)}\right)\eta\nonumber\\
	&+
	\sum_{k=0}^{\infty}\sum_{m=0}^{\infty}E^{-1}Q_{k,m}^{A,B}\frac{t^{k(\alpha-\beta)+m\alpha+1}}{\Gamma(k(\alpha-\beta)+m\alpha+2)}\hat{\eta}\nonumber\\
	&+\int\limits_{0}^{t}\sum_{k=0}^{\infty}\sum_{m=0}^{\infty}E^{-1}Q_{k,m}^{A,B}\frac{(t-s)^{k(\alpha-\beta)+m\alpha+\alpha-1}}{\Gamma(k(\alpha-\beta)+m\alpha+\alpha)}g(s)\mathrm{d}s\nonumber\\
	&\coloneqq\left( E^{-1}+t^{\alpha}E^{-1}\mathscr{E}_{\alpha-\beta,\alpha,\alpha+1}^{A,B}(t)B\right) \eta+tE^{-1}\mathscr{E}_{\alpha-\beta,\alpha,2}^{A,B}(t)\hat{\eta}\nonumber\\&+\int\limits_{0}^{t}(t-s)^{\alpha-1}E^{-1}\mathscr{E}_{\alpha-\beta,\alpha,\alpha}^{A,B}(t-s)g(s)\mathrm{d}s, \quad t>0.
	\end{align}
\end{theorem}

\begin{remark}
Let $A_{0}=\Theta$ be a zero operator. Then, a mild solution $y(\cdot)\in \mathbb{C}^{2}(\mathbb{J},X)$ of the Cauchy problem \eqref{mtde-3} 
\begin{equation}\label{mtde-3}
\begin{cases}
\left( \prescript{C}{}{D^{\alpha}_{0+}}Ey\right)(t)-B_{0}y(t)=g(t), \quad  t>0, \quad 1<\alpha\leq2,\\
y(0)=\eta , \quad y^{\prime}(0)=\tilde{\eta},
\end{cases}
\end{equation}
can be determined by means of two parameter Mittag-Leffler functions as follows
\begin{align}
y(t)&=\sum_{m=0}^{\infty}E^{-1}B^{m}\frac{t^{m\alpha}}{\Gamma(m\alpha+1)}\eta+
\sum_{m=0}^{\infty}E^{-1}B^{m}\frac{t^{m\alpha+1}}{\Gamma(m\alpha+2)}\hat{\eta}
+\int\limits_{0}^{t}\sum_{m=0}^{\infty}E^{-1}B^{m}\frac{(t-s)^{m\alpha+\alpha-1}}{\Gamma(m\alpha+\alpha)}g(s)\mathrm{d}s\nonumber\\
&\coloneqq E^{-1}E_{\alpha,1}(Bt^{\alpha})\eta+tE^{-1}E_{\alpha,2}(Bt^{\alpha})\hat{\eta}+\int\limits_{0}^{t}(t-s)^{\alpha-1}E^{-1}E_{\alpha,\alpha}(B(t-s)^{\alpha})g(s)\mathrm{d}s, \quad t>0.
\end{align}
\end{remark}

Similar problem to \eqref{mtde-3} has been considered in \cite{Feckan-1} for Sobolev type functional evolution equations with fractional-order as follows:
\begin{equation}\label{mtde-4}
\begin{cases}
 \prescript{C}{0}{D}^{q}_{t}(Ex(t))+Ax(t)=f(t,x_{t}), \quad  t\in J\coloneqq[0,a],\\
x(t)=\phi(t) , \quad -r\leq t\leq 0,
\end{cases}
\end{equation}

where $\prescript{C}{0}{D}^{q}_{t}$ is the Caputo fractional derivative of order $0<q<1$ with lower limit zero. The operators $A: D(A)\subset X \to Y$ and $E:D(E)\subset X \to Y$, where $X,Y$are Banach spaces. Moreover, $f(\cdot,\cdot):J\times C\to Y$ with $C\coloneqq C\left([-r,0],X\right)$. $x(\cdot):J^{*}\coloneqq[-r,a]\to X$ is continuous, $x_{t}$ is the element of $C$ defined by $x_{t}(s)\coloneqq x(t+s)$, $-r\leq s\leq 0$. The domain $D(E)$ of $E$ becomes a Banach space with norm $\|x\|_{D(E)}\coloneqq\|Ex\|_{Y}$, $x\in D(E)$ and $\phi\in C(E)\coloneqq C([-r,0],D(E))$.

Feckan et al. \cite{Feckan-1} have introduced the following assumptions on the operators $A$ and $E$:

$(\hat{H}_{1})$: $A$ and $E$ are linear operators and $A$ is closed;

$(\hat{H}_{2})$: $D(E)\subset D(A)$ and $E$ is bijective;

$(\hat{H}_{3})$: Linear operator $E^{-1}: Y \to D(E)\subset X$ is compact.

By making use of the substitution $x(t)=E^{-1}y(t)$, under the hypotheses $(\hat{H}_{1})$-$(\hat{H}_{3})$, we transform the Sobolev type fractional-order functional evolution system \eqref{mtde-4} to the  following evolution system with a linear bounded operator $\hat{A}\coloneqq-AE^{-1}:Y\to Y$:

\begin{equation}\label{mtde-5}
\begin{cases}
\prescript{C}{0}{D^{\alpha}_{t}}y (t) -\hat{A}y(t)=f(t,E^{-1}y_{t}), \quad t\in J,\\
y(t)=\varphi(t) , \quad -r\leq t\leq 0,
\end{cases}
\end{equation}
where $y(\cdot):J^{*}\to Y$, $\varphi(t)=E\phi(t), t \in [-r,0]$ and $y_{t}(s)=y(t+s), s\in[-r,0]$.

A mild solution of an initial value problem for functional evolution equation \eqref{mtde-5} can be expressed by means of classical Mittag-Leffler type functions as

\begin{align}
y(t)=E_{q,1}(\hat{A}t^{q})\varphi(0)+\int\limits_{0}^{t}(t-s)^{q-1}E_{q,q}(\hat{A}(t-s)^{q})f(s,E^{-1}y_{s})\mathrm{d}s, \quad t>0.
\end{align}
Thus, the mild solution of Sobolev type functional evolution equation of fractional-order should be represented by 
\begin{align}\label{correct}
x(t)=E^{-1}E_{q,1}(\hat{A}t^{q})E\phi(0)+\int\limits_{0}^{t}(t-s)^{q-1}E^{-1}E_{q,q}(\hat{A}(t-s)^{q})f(s,x_{s})\mathrm{d}s,\quad t>0.
\end{align}
However, the mild solution of \eqref{mtde-4} was represented via characteristic solution operators (see Lemma 3.1 in \cite{Feckan-1}) instead of Mittag-Leffler functions generated by a linear operator $\hat{A}\coloneqq-AE^{-1}\in \mathscr{B}(Y)$. It should be stressed out that if $E=I$, then a mild solution of \eqref{mtde-4} can be expressed with the help of characteristic solution operators (see Remark 3.1 in \cite{Feckan-1}), otherwise, under hypotheses $(\hat{H}_{1})-(\hat{H}_{3})$ it should be determined by classical Mittag-Leffler functions with two parameters  which are compact linear operators in $Y$ as \eqref{correct}. 
\begin{remark}
	In particular case, we consider the following initial value problem for multi-dimensional multi-term fractional differential equation with noncommutative matrices
	\begin{equation}\label{multi-1}
	\begin{cases}
	\left( \prescript{C}{}{D^{\alpha}_{0+}}y\right) (t) -A_{0} \left( \prescript{C}{}{D^{\beta}_{0+}}y\right) (t)-B_{0}y(t)=g(t), \quad  t>0,\\
	y(0)=\eta , \quad y^{\prime}(0)=\tilde{\eta},
	\end{cases}
	\end{equation}
	where $\prescript{C}{}{D^{\alpha}_{0+}}$ and $ \prescript{C}{}{D^{\beta}_{0+}}$ Caputo fractional derivatives of orders $1<\alpha\leq2$ and $0<\beta\leq1$, respectively, with the lower limit zero. $E=I\in\mathbb{R}^{n\times n}$ is an identity matrix,  the matrices $A_{0},B_{0}\in \mathbb{R}^{n\times n}$ are nonpermutable i.e., $AB\neq BA$, $y(t)\in \mathbb{R}^{n}$ is a vector-valued function on $\mathbb{J}$, i.e., $y(\cdot):\mathbb{J}\to \mathbb{R}^{n}$ and $\eta, \hat{\eta}\in \mathbb{R}^{n}$. In addition, a forced term $g(\cdot): \mathbb{J}\to \mathbb{R}^{n}$ is a continuous function.
	
	The exact analytical representation of solution $y(\cdot)\in \mathbb{C}^{2}(\mathbb{J},\mathbb{R}^{n})$ of \eqref{multi-1} can be expressed by
	
	\begin{align}
	y(t)&\coloneqq\left( 1+t^{\alpha}\mathscr{E}_{\alpha-\beta,\alpha,\alpha+1}^{A_{0},B_{0}}(t)B_{0}\right) \eta+t\mathscr{E}_{\alpha-\beta,\alpha,2}^{A_{0},B_{0}}(t)\hat{\eta}+\int\limits_{0}^{t}(t-s)^{\alpha-1}\mathscr{E}_{\alpha-\beta,\alpha,\alpha}^{A_{0},B_{0}}(t-s)g(s)\mathrm{d}s\nonumber\\&=\left(1+ \sum_{k=0}^{\infty}\sum_{m=0}^{\infty}Q_{k,m}^{A_{0},B_{0}}B_{0}\frac{t^{k(\alpha-\beta)+m\alpha+\alpha}}{\Gamma(k(\alpha-\beta)+m\alpha+\alpha+1)}\right)\eta+
	\sum_{k=0}^{\infty}\sum_{m=0}^{\infty}Q_{k,m}^{A_{0},B_{0}}\frac{t^{k(\alpha-\beta)+m\alpha+1}}{\Gamma(k(\alpha-\beta)+m\alpha+2)}\hat{\eta}\nonumber\\
	&+\int\limits_{0}^{t}\sum_{k=0}^{\infty}\sum_{m=0}^{\infty}Q_{k,m}^{A_{0},B_{0}}\frac{(t-s)^{k(\alpha-\beta)+m\alpha+\alpha-1}}{\Gamma(k(\alpha-\beta)+m\alpha+\alpha)}g(s)\mathrm{d}s, \quad t>0.
	\end{align}
\end{remark}

\section{A representation of solutions of \eqref{mtde} with permutable linear operators}\label{sec:4}
To get an analytical representation of a mild solution of \eqref{mtde-1} with permutable linear operators i.e., $AB=BA$, first, we need to prove auxiliary lemma for making use of Laplace integral transform
according to the Theorem \ref{thm2}. Moreover, the scalar analogue of following theorem has been considered by Ahmadova and Mahmudov for fractional Langevin equations with constant coefficients in \cite{Ahmadova-Mahmudov}. In general, the following theorem is true for $\alpha>0$, $\alpha>\beta$ and $\gamma\in \mathbb{R}$.
\begin{theorem} \label{Q^A,B-per}
	Let $m \in \mathbb{N}_{0}$ and $Re(s)>0$. For $A,B\in\mathscr{B}(Y)$ with $[A,B]=AB-BA=0$, we have:
	\begin{align*}
	\mathscr{L}^{-1}\Bigl\{ \frac{s^{\gamma}B^{m}}{(s^{\alpha}I-A s^{\beta})^{m+1}}\Bigr\}(t)&= t^{m\alpha+\alpha+\gamma-1}\sum_{k=0}^{\infty}\binom{k+m}{m}\frac{A^{k}B^{m}t^{k(\alpha-\beta)}}{\Gamma(k(\alpha-\beta)+m\alpha+\alpha-\gamma)}\\
	&=t^{m\alpha+\alpha-\gamma-1}E^{m+1}_{\alpha-\beta, m\alpha+\alpha-\gamma}(A t^{\alpha-\beta})B^{m}. %\\
	%&\coloneqq t^{(m+1)\alpha-\gamma-1}E^{m+1}_{\alpha-\beta, (m+1)\alpha}(A T^{\alpha-\beta})B^{m}.
	\end{align*}
\end{theorem}

\begin{proof}
	By using the Taylor series representation of $\frac{1}{(1-t)^{m+1}}, m \in \mathbb{N}_{0}$ of the form
	\begin{equation*}
	\frac{1}{(1-t)^{m+1}}= \sum_{k=0}^{\infty}\binom{k+m}{m}t^{k}, \quad |t|<1,
	\end{equation*}
	we achieve that
	\begin{align*}
	\frac{s^{\gamma}B^{m}}{(s^{\alpha}I-A s^{\beta})^{m+1}}=\frac{s^{\gamma}B^{m}}{(s^{\alpha }I)^{m+1}}\frac{1}{(1-\frac{A}{s^{\alpha-\beta}} )^{m+1}}&=\frac{s^{\gamma}B^{m}}{s^{(m+1)\alpha}}\sum_{k=0}^{\infty}\binom{k+m}{m}\Big(\frac{A}{s^{\alpha-\beta}}\Big)^{k}\\
	&=\sum_{k=0}^{\infty}\binom{k+m}{m}\frac{A^{k}B^{m}}{s^{(m+1)\alpha+k(\alpha-\beta)-\gamma}}.
	\end{align*}
	By using the inverse Laplace integral formula for the above function, we get the desired result:
	\begin{align*}
	\mathscr{L}^{-1}\Bigl\{ \frac{s^{\gamma}B^{m}}{(s^{\alpha}I-A s^{\beta})^{m+1}}\Bigr\}(t)&=\sum_{k=0}^{\infty}A^{k}B^{m}\binom{k+m}{m}\mathscr{L}^{-1}\Bigr\{\frac{1}{s^{k(\alpha-\beta)+(m+1)\alpha-\gamma}}\Bigr\}(t)\\
	&=\sum_{k=0}^{\infty}A^{k}B^{m}\binom{k+m}{m}\frac{t^{k(\alpha-\beta)+m\alpha+\alpha-\gamma-1}}{\Gamma(k(\alpha-\beta)+m\alpha+\alpha-\gamma)}\\
	&=t^{m\alpha+\alpha-\gamma-1}E^{m+1}_{\alpha-\beta, m\alpha+\alpha-\gamma}(A t^{\alpha-\beta})B^{m}.
	\end{align*}
	We have required an extra condition on $s$ such that
	\begin{equation*}
	 s^{\alpha-\beta}>\|A\|,
	\end{equation*}
	for proper convergence of the series. But, this condition can be removed at the end of calculation since analytic
	continuation of both sides, to give the desired result for all $ s \in \mathbb{C}$ which is satisfying $Re(s)>0$.
\end{proof}

Then, we acquire analytical representation of mild solution for multi-term fractional evolution equation with permutable linear bounded operators via the following theorem.

\begin{theorem}	Let $A,B\in\mathscr{B}(Y)$ with zero commutator, i.e., $\left[ A, B\right] \coloneqq AB- BA = 0$. Assume that $g(\cdot): \mathbb{J} \to X$ and $\left(\prescript{C}{}{D^{\beta}_{0^{+}}x}\right) (\cdot)$ for $0<\beta\leq1$ are exponentially bounded.
	A mild solution $x(\cdot)\in \mathbb{C}^{2}(\mathbb{J},Y)$ of the Cauchy problem \eqref{mtde-1} can be represented by means of bivariate Mittag-Leffler type functions \eqref{bivtype} as follows
	\begin{align}
	x(t)&=\left( I+t^{\alpha}BE_{\alpha-\beta,\alpha,\alpha+1}(At^{\alpha-\beta},Bt^{\alpha})\right) \eta+tE_{\alpha-\beta,\alpha,2}(At^{\alpha-\beta},Bt^{\alpha})\hat{\eta}+t^{\alpha-1}E_{\alpha-\beta,\alpha,\alpha}(At^{\alpha-\beta},Bt^{\alpha})\ast g(t)\nonumber\\&=\left(I+ \sum_{k=0}^{\infty}\sum_{m=0}^{\infty}\binom{k+m}{m}A^{k}B^{m+1}\frac{t^{k(\alpha-\beta)+m\alpha+\alpha}}{\Gamma(k(\alpha-\beta)+m\alpha+\alpha+1)}\right)\eta\nonumber\\&+
	\sum_{k=0}^{\infty}\sum_{m=0}^{\infty}\binom{k+m}{m}A^{k}B^{m}\frac{t^{k(\alpha-\beta)+m\alpha+1}}{\Gamma(k(\alpha-\beta)+m\alpha+2)}\hat{\eta}\nonumber\\
	&+\int\limits_{0}^{t}\sum_{k=0}^{\infty}\sum_{m=0}^{\infty}\binom{k+m}{m}A^{k}B^{m}\frac{(t-s)^{k(\alpha-\beta)+m\alpha+\alpha-1}}{\Gamma(k(\alpha-\beta)+m\alpha+\alpha)}g(s)\mathrm{d}s, \quad t>0.
	\end{align}
\end{theorem}
\begin{proof}
	We recall that the existence of Laplace transform of $x(\cdot)$ and its Caputo derivatives $\left(\prescript{C}{}{D^{\alpha}_{0^{+}}x}\right) (\cdot)$ and $\left(\prescript{C}{}{D^{\beta}_{0^{+}}x}\right) (\cdot)$  for $ 1<\alpha\leq 2$ and $0<\beta\leq 1$, respectively, is guaranteed by Theorem \ref{thm2}.
	Thus, to find the mild solution $x(t)$ of \eqref{mtde} with permutable linear operators, i.e., $AB=BA$, we can use the Laplace transform technique. By assuming $T=\infty$, applying the Laplace transform technique on both sides of equation \eqref{mtde-1} and solving the equation with respect to the $X(s)$, we get 
	\begin{align*}
	X(s)
	&=s^{-1}\eta+s^{-1}\left( s^{\alpha}I-As^{\beta}-B\right)^{-1}B\eta+s^{\alpha-2} \left( s^{\alpha}I-As^{\beta}-B\right)^{-1}\hat{\eta}\\&+\left( s^{\alpha-1}I-As^{\beta}-B\right)^{-1}G(s).
	\end{align*}
	
	On the other hand, since \eqref{operator} for sufficiently large $s$, we have
	\begin{equation*}
	\|\left( s^{\alpha}I-As^{\beta}\right) ^{-1}B\|<1.
	\end{equation*}
	
	Then, for permutable linear operators $A,B\in \mathscr{B}(Y)$ and sufficiently large $s$, one can attain
	\begin{align*}
	\left( s^{\alpha}I-As^{\beta}-B\right)^{-1}
	&=\left(s^{\alpha}I-A s^{\beta}\right)^{-1}\left( I-\left( s^{\alpha}I-As^{\beta}\right) ^{-1}B\right)^{-1} \\
	&=\left(s^{\alpha}I-A s^{\beta}\right)^{-1}\sum_{m=0}^{\infty} \left(s^{\alpha}I-As^{\beta} \right)^{-m}B^{m}\\
	&=\sum_{m=0}^{\infty}\frac{B^{m}}{\left(s^{\alpha}I-As^{\beta} \right)^{(m+1)}}.
	\end{align*}
	Then taking inverse Laplace transform, we have
	\allowdisplaybreaks
	\begin{align} \label{y(t)}
	x(t)=\mathscr{L}^{-1}\left\lbrace s^{-1}\right\rbrace(t) \eta&+ \mathscr{L}^{-1}\left\lbrace\sum_{m=0}^{\infty}\frac{s^{-1}B^{m+1}}{\left(s^{\alpha}I-As^{\beta} \right)^{(m+1)}}\right\rbrace (t)B\eta \nonumber\\
	&+\mathscr{L}^{-1}\left\lbrace\sum_{m=0}^{\infty}\frac{s^{\alpha-2}B^{m}}{\left(s^{\alpha}I-As^{\beta} \right)^{(m+1)}}\right\rbrace (t)\hat{\eta}\nonumber\\
	&+\mathscr{L}^{-1}\left\lbrace \sum_{m=0}^{\infty}\frac{B^{m}}{\left(s^{\alpha}I-As^{\beta} \right)^{(m+1)}} G(s)\right\rbrace (t).
	\end{align}
	Therefore, in accordance with Theorem \eqref{Q^A,B-per}, we acquire
	\begin{align}
	x(t)&=\left\lbrace I+ \sum_{k=0}^{\infty}\sum_{m=0}^{\infty}\binom{k+m}{k}\frac{A^{k}B^{m+1}t^{k(\alpha-\beta)+m\alpha+\alpha}}{\Gamma(k(\alpha-\beta)+m\alpha+\alpha+1)}\right\rbrace \eta\nonumber\\&+
	\sum_{k=0}^{\infty}\sum_{m=0}^{\infty}\binom{k+m}{k}\frac{A^{k}B^{m}t^{k(\alpha-\beta)+m\alpha+1}}{\Gamma(k(\alpha-\beta)+m\alpha+2)}\hat{\eta}\nonumber\\
	&+\int\limits_{0}^{t}\sum_{k=0}^{\infty}\sum_{m=0}^{\infty}\binom{k+m}{k}\frac{A^{k}B^{m}(t-s)^{k(\alpha-\beta)+m\alpha+\alpha-1}}{\Gamma(k(\alpha-\beta)+m\alpha+\alpha)}g(s)\mathrm{d}s\nonumber\\
	&\coloneqq\left( I+t^{\alpha}BE_{\alpha-\beta,\alpha,\alpha+1}(At^{\alpha-\beta},Bt^{\alpha})\right) \eta+tE_{\alpha-\beta,\alpha,2}(At^{\alpha-\beta},Bt^{\alpha})\hat{\eta}\nonumber\\&+\int\limits_{0}^{t}(t-s)^{\alpha-1}E_{\alpha-\beta,\alpha,\alpha}(A(t-s)^{\alpha-\beta},B(t-s)^{\alpha})g(s)\mathrm{d}s, \quad t>0.	
	\end{align}
\end{proof}
\begin{remark}
The analytical mild solution for the initial value problem for \eqref{mtde-1} can be attained from the property of $Q_{k,m}^{A,B}$ \eqref{commutative} for linear bounded operators $A,B\in\mathscr{B}(Y)$ satisfying $AB=BA$ where
\begin{equation*}
Q_{k,m}^{A,B}=\binom{k+m}{m}A^{k}B^{m}, \quad  k,m \in \mathbb{N}_{0}.
\end{equation*}
\end{remark}

It should be emphasized that the assumption on the exponential boundedness of the function $g(\cdot)$ and $\left( \prescript{C}{}{D^{\beta}_{0+}}x\right) (\cdot)$ for $0<\beta\leq 1$ ($\left( \prescript{C}{}{D^{\alpha}_{0+}}x\right) (\cdot)$ for $1<\alpha\leq 2$ ) can be omitted for the case of permutable linear bounded operators, too.

\begin{theorem}
	Let $A,B\in\mathscr{B}(Y)$ with zero commutator, i.e., $\left[ A, B\right] \coloneqq AB- BA= 0$. 
	A mild solution $x(\cdot)\in \mathbb{C}^{2}(\mathbb{J},Y)$ of the Cauchy problem \eqref{mtde-1} can be expressed as
	\begin{align}\label{form-per}
	x(t)&=\left(I+ \sum_{k=0}^{\infty}\sum_{m=0}^{\infty}\binom{k+m}{m}A^{k}B^{m+1}\frac{t^{k(\alpha-\beta)+m\alpha+\alpha}}{\Gamma(k(\alpha-\beta)+m\alpha+\alpha+1)}\right)\eta\nonumber\\&+
	\sum_{k=0}^{\infty}\sum_{m=0}^{\infty}\binom{k+m}{m}A^{k}B^{m}\frac{t^{k(\alpha-\beta)+m\alpha+1}}{\Gamma(k(\alpha-\beta)+m\alpha+2)}\hat{\eta}\nonumber\\
	&+\int\limits_{0}^{t}\sum_{k=0}^{\infty}\sum_{m=0}^{\infty}\binom{k+m}{m}A^{k}B^{m}\frac{(t-s)^{k(\alpha-\beta)+m\alpha+\alpha-1}}{\Gamma(k(\alpha-\beta)+m\alpha+\alpha)}g(s)\mathrm{d}s, \quad t>0.
	\end{align}
\end{theorem}

\begin{proof}
For linear homogeneous and inhomogeneous cases, by using the following Pascal identity for binomial coefficients:
\begin{equation*}
\binom{k+m}{m}=\binom{k+m-1}{m}+\binom{k+m-1}{m-1}, \quad  k, m\in \mathbb{N},
\end{equation*}	
 the formula \eqref{formula} and fractional Leibniz integral rules \eqref{Leibniz}, it can be easily shown that \eqref{form-per} is a mild solution of the Cauchy problem for \eqref{mtde-1} with permutable linear bounded operators. Moreover, this case have considered by Mahmudov et al. for multi-dimensional Bagley-Torvik equations with permutable matrices in \cite{Mahmudov-Huseynov-Aliev-Aliev}.
\end{proof}

\begin{theorem}	Let $A,B\in\mathscr{B}(Y)$ with zero commutator, i.e., $\left[ A, B\right] \coloneqq AB- BA = 0$. A mild solution $y(\cdot)\in \mathbb{C}^{2}(\mathbb{J},X)$ of the Cauchy problem \eqref{mtde} can be determined as below
	\begin{align}
	y(t)&=\left( E^{-1}+t^{\alpha}E^{-1}BE_{\alpha-\beta,\alpha,\alpha+1}(At^{\alpha-\beta},Bt^{\alpha})\right) \eta+tE^{-1}E_{\alpha-\beta,\alpha,2}(At^{\alpha-\beta},Bt^{\alpha})\hat{\eta}\nonumber\\&+t^{\alpha-1}E^{-1}E_{\alpha-\beta,\alpha,\alpha}(At^{\alpha-\beta},Bt^{\alpha})\ast g(t)\nonumber\\&=\left(E^{-1}+ \sum_{k=0}^{\infty}\sum_{m=0}^{\infty}\binom{k+m}{m}E^{-1}A^{k}B^{m+1}\frac{t^{k(\alpha-\beta)+m\alpha+\alpha}}{\Gamma(k(\alpha-\beta)+m\alpha+\alpha+1)}\right)\eta\nonumber\\&+
	\sum_{k=0}^{\infty}\sum_{m=0}^{\infty}\binom{k+m}{m}E^{-1}A^{k}B^{m}\frac{t^{k(\alpha-\beta)+m\alpha+1}}{\Gamma(k(\alpha-\beta)+m\alpha+2)}\hat{\eta}\nonumber\\
	&+\int\limits_{0}^{t}\sum_{k=0}^{\infty}\sum_{m=0}^{\infty}\binom{k+m}{m}E^{-1}A^{k}B^{m}\frac{(t-s)^{k(\alpha-\beta)+m\alpha+\alpha-1}}{\Gamma(k(\alpha-\beta)+m\alpha+\alpha)}g(s)\mathrm{d}s, \quad t>0.
	\end{align}
\end{theorem}

\begin{remark}
	In a special case, the exact analytical representation of solution $y(\cdot)\in \mathbb{C}^{2}(\mathbb{J},\mathbb{R}^{n})$ of Cauchy problem for multi-dimensional  fractional differential equation with multi-orders and permutable matrices $A_{0},B_{0} \in  \mathbb{R}^{n\times n}$  i.e., $A_{0}B_{0}= B_{0}A_{0}$ \eqref{multi-1} can be represented by,
	\begin{align}
y(t)&=\left(1+ \sum_{k=0}^{\infty}\sum_{m=0}^{\infty}\binom{k+m}{m}A_{0}^{k}B_{0}^{m+1}\frac{t^{k(\alpha-\beta)+m\alpha+\alpha}}{\Gamma(k(\alpha-\beta)+m\alpha+\alpha+1)}\right)\eta\nonumber\\&+
\sum_{k=0}^{\infty}\sum_{m=0}^{\infty}\binom{k+m}{m}A_{0}^{k}B_{0}^{m}\frac{t^{k(\alpha-\beta)+m\alpha+1}}{\Gamma(k(\alpha-\beta)+m\alpha+2)}\hat{\eta}\nonumber\\
&+\int\limits_{0}^{t}\sum_{k=0}^{\infty}\sum_{m=0}^{\infty}\binom{k+m}{m}A_{0}^{k}B_{0}^{m}\frac{(t-s)^{k(\alpha-\beta)+m\alpha+\alpha-1}}{\Gamma(k(\alpha-\beta)+m\alpha+\alpha)}g(s)\mathrm{d}s, \quad t>0.
\end{align}
\end{remark}

\section{Discussion and future work}\label{sec:concl}
In this research work, first, we convert Sobolev type fractional  evolution equation with multi-orders \eqref{mtde} to multi-term fractional evolution equation with linear bounded operators \eqref{mtde-1}. Secondly, we give the sufficient conditions to guarantee the rationality of solving multi-term fractional differential equations with linear bounded operators by the Laplace transform method. Then we solve linear inhomogeneous fractional evolution equation with nonpermutable \& permutable linear bounded operators $A,B\in\mathscr{B}(Y)$ by making use of Laplace integral transform. 
Next we  propose exact analytical representation of a mild solution of \eqref{mtde-1} and \eqref{mtde}, respectively with the help of newly defined Mittag-Leffler type function which is generated by linear bounded operators by removing the strong condition which is an exponential boundedness of  a forced term and one of fractional orders with the help of analytical methods, namely:  verification by substitution and fractional analogue of variation of constants formula.

	The main contributions of this paper are as follows:

\begin{itemize}
	\item we introduce a new Mittag-Leffler type function which is generated by linear bounded operators $A,B\in\mathscr{B}(Y)$ via a double infinity series ;
	\item we derive new properties of Mittag-Leffler type function which are useful tool for checking the candidate solutions of multi-term fractional differential equations;
	\item we propose the property of $Q_{k,m}^{A,B}$ with nonpermutable linear operators $A,B\in \mathscr{B}(Y)$ which is a generalization of well-known Pascal's rule binomial coefficients.
	\item we acquire the analytical representation of a mild solution for linear Sobolev type fractional  multi-term  evolution equations with nonpermutable  and permutable linear operators;
	\item we derive the exact analytical representation of multi-dimensional fractional differential equations with two independent orders and  nonpermutable \& permutable matrices. 
\end{itemize}

The possible directions for future work in which to extend the results of this paper is  looking at Sobolev type fractional  functional  evolution equations with multi-orders \eqref{f-w}.
Furthermore, one can expect the results of this paper to hold for a class of problems such as Sobolev type functional evolution system governed by
\begin{equation}\label{f-w}
\begin{cases}
\left( \prescript{C}{}{D^{\alpha}_{0+}}Ey\right) (t) -A_{0} \left( \prescript{C}{}{D^{\beta}_{0+}}y\right) (t)=B_{0}y(t-\tau)+g(t), \quad 1\geq \alpha >\beta >0, \\
Ey(\tau)=E\phi(t),\quad -\tau \leq t \leq 0.
\end{cases}
\end{equation}
It would be interesting to see how the theorems
proved above can be extended to these cases. Another direction in which we would like to investigate stability and approximate controllability results for Sobolev type multi-term fractional differential equations \eqref{mtde} and \eqref{f-w}.

\end{document}